\input epsf

\documentclass[10pt]{article}

\usepackage{latexsym,amsmath,amsthm,amssymb,amsfonts,amscd,multirow}
\usepackage{epsfig,subfigure,verbatim,epstopdf,graphics}
\usepackage{color}

\newtheorem{theorem}{Theorem}[section]
\newtheorem{proposition}[theorem]{Proposition}

\newtheorem{lemma}[theorem]{Lemma}

\newtheorem{remark}[theorem]{Remark}

\newcommand{\beq}{\begin{equation}}
\newcommand{\eeq}{\end{equation}}
\newcommand{\beqa}{\begin{eqnarray}}
\newcommand{\eeqa}{\end{eqnarray}}
\newcommand{\bit}{\begin{itemize}}
\newcommand{\eit}{\end{itemize}}
\newcommand{\bedef}{\begin{defn}}
\newcommand{\edefn}{\end{defn}}
\newcommand{\bpro}{\begin{prop}}
\newcommand{\epro}{\end{prop}}

\newcommand{\mE}{{\mathcal E}}
\newcommand{\mT}{{\mathcal T}}
\newcommand{\mG}{{\mathcal G}}
\newcommand{\mU}{{\mathcal U}}
\newcommand{\OO}{{\Omega}}
\newcommand{\Ox}{{\Omega_x}}
\newcommand{\Ov}{{\Omega_v}}

\newcommand{\Kx}{{K_x}}

\newcommand{\stab}{{\bf{STAB}}}
\newcommand{\ene}{{\bf N}}

\title{Error Estimates of Runge-Kutta Discontinuous Galerkin Methods for the Vlasov-Maxwell System }

\author{He Yang and Fengyan Li \footnote{This research is partially supported by NSF CAREER award DMS-0847241 and DMS-1318409. Address: Department of Mathematical Sciences, Rensselaer Polytechnic Institute, Troy, NY   12180-3590, United States.
Email: yangh8@rpi.edu, lif@rpi.edu. }}

\begin{document}


\baselineskip=13pt


\maketitle


\begin{abstract}

In this paper, error analysis is established for Runge-Kutta discontinuous Galerkin (RKDG) methods to solve the Vlasov-Maxwell system.
This nonlinear hyperbolic system  describes the time evolution of collisionless plasma particles of a single species under the self-consistent electromagnetic field, and it models
many phenomena in both laboratory and astrophysical plasmas.
The methods involve a third order TVD Runge-Kutta discretization in time and upwind discontinuous Galerkin discretizations of arbitrary order in phase domain.
With the assumption that the exact solution has sufficient regularity, the $L^2$ errors of the particle number density function as well as electric and magnetic fields at any given time $T$ are
bounded by $C h^{k+\frac{1}{2}}+C\tau^3$ under a CFL condition
$\tau /h \leq \gamma$. Here $k$ is the polynomial degree used in phase space discretization,
satisfying $k \geq \left \lceil \frac{d_x + 1}{2} \right \rceil$ (the smallest integer greater than or equal to $\frac{d_x+1}{2}$, with
 $d_x$ being the dimension of spatial domain),
$\tau$ is the time step, and $h$ is the maximum mesh size in phase space.  Both $C$ and $\gamma$ are positive constants independent of $h$ and $\tau$, and they may depend on
the polynomial degree $k$, time $T$, the size of the phase domain,  certain mesh parameters, and some Sobolev norms of the exact solution. The analysis can
be extended to RKDG methods with other numerical fluxes and to RKDG methods solving relativistic Vlasov-Maxwell equations.

\end{abstract}


{\bf AMS(MOS) subject classification:} 65M15, 	65M60 , 65M06,   35Q83,  35L50

\section{Introduction}
\label{intro}

In this paper, we will establish error estimates of the Runge-Kutta discontinuous Galerkin (RKDG) methods for solving the dimensionless Vlasov-Maxwell (VM) equations
\begin{equation}
 \label{eq1:vm1}
 \left\{
 \begin{aligned}
 &\partial_t f + v\cdot \nabla_x f + (E+v\times B)\cdot \nabla_v f = 0, \\
 &\partial_t E = \nabla \times B -J, \quad \partial_t B = -\nabla \times E,\\
 & \nabla \cdot E = \rho - \rho_i, \quad \nabla \cdot B = 0,
 \end{aligned}
 \right.
\end{equation}
with
\begin{equation}
 \label{eq2:vm2}
  \rho(x,t)=\int_{\Omega_v}f(x,v,t)dv, \quad J(x,t)=\int_{\Omega_v}f(x,v,t)vdv.
\end{equation}
This system  describes the time evolution of collisionless plasma particles of a single species, such as electrons or ions, under the self-consistent electromagnetic field.
Here $f(x,v,t) \geq 0$ is the particle number density function in the phase space with $(x, v)$ at time $t$, $E(x,t)$ is the electric field, $B(x,t)$ is the magnetic field,
$J(x,t)$ is the current density, $\rho(x,t)$ is the charge density, and $\rho_i$ is the charge density of the background particles.
The system \eqref{eq1:vm1} is defined  on the phase domain $\Omega = \Omega_x \times \Omega_v$, where $\Omega_x = [L_{x,1}, L_{x,2}]^{d_x}$ is the spatial domain and
$\Omega_v=[L_{v,1}, L_{v,2}]^{d_v}$ is the velocity domain ($d_x, d_v = 1,2$ or $3$), with periodic boundary conditions in $x$. In $v$ direction, $f$ is assumed to have compact support.
We further assume the VM system is globally neutral, i.e. $\int_{\Omega_x} (\rho - \rho_i) dx = 0$. Note that this is compatible with the periodic boundary conditions in $x$.

The VM equations model many phenomena in both laboratory and astrophysical plasmas, and accurate and reliable numerical simulation of this system has fundamental importance.
 Particle methods \cite{BL1985, Cottet-Raviart, Jacob-Hesthaven} have been widely used since 60's because of their  low computational cost especially when the dimension of the phase space is high, yet with their numerical noise, it is hard for the methods to produce very accurate approximations. In recent years, many high order Eulerian methods have been developed in the context of Vlasov-Poisson equations or the VM equations. Some examples include semi-Lagrangian methods \cite{CK1976, SRBG1999, PM2008, Qiu-Shu:WENO}, continuous finite element methods \cite{ZGB1988, ZBG1988}, and methods based on Fourier transform \cite{KF1994, Eliasson:2003,Eliasson:2011}.

In \cite{CGLM}, one of the authors and her collaborators proposed and analyzed semi-discrete discontinuous Galerkin (DG) methods for the VM system.
The methods were further combined with Runge-Kutta time discretizations, resulting in Runge-Kutta discontinuous Galerkin (RKDG) methods, and their performance in accuracy, stability, and conservation
 was demonstrated numerically.
Note that DG methods were previously proposed and studied in \cite{Heath-Gamba-Morrison-Michler,Ayuso-Carrillo-Shu:1d,Ayuso-Carrillo-Shu:multid,Cheng-Gamba:2011}
for the Vlasov-Poisson system.
 DG discretizations are chosen for the phase domain in \cite{CGLM} due to their high accuracy, compactness,
high efficiency in parallel implementations,  flexibility with complicated geometry as well as boundary conditions and adaptive simulations.
All aforementioned properties make the methods a competitive candidate to simulate the VM system accurately with reasonable computational cost
especially for lower dimensional cases (e.g. the 1D2V or 2D2V system).
This is even so if one further makes good use of the modern computer architectures. More importantly,
with properly designed numerical fluxes and up to some boundary effect, semi-discrete DG methods have {\em provable} conservation property
of both mass and total energy, and this is shared by very few high order methods.
  On the other hand, though RKDG methods have been widely used in many applications since they were introduced  \cite{CS1991, CS1989},
 theoretical analysis for such fully discrete methods is relatively little.
 Error estimations based on Fourier analysis and some symbolic computations were carried out in \cite{ZhongShu:2011} and \cite{YangLiQiu}
 when RKDG methods are applied to linear problems on uniform meshes. Important developments were made by Zhang and Shu in  \cite{ZS2004, ZS2006, ZS2010},
  where error estimates to the smooth solutions on uniform or non-uniform meshes were developed for RKDG methods with the second order RK time discretization for scaler \cite{ZS2004} and
  symmetrizable conservation laws \cite{ZS2006}, and with  the third order RK method for scaler conservations laws in \cite{ZS2010}. In \cite{ZS2010},
  $L^2$-norm stability was also obtained for linear conservation laws, and such analysis so far is unavailable for nonlinear cases.
In this paper, we use the idea in \cite{ZS2010} to obtain the error estimates of the fully discrete RKDG methods for the VM system when the exact solutions have sufficient regularity.
 In particular, a third order TVD Runge-Kutta method \cite{ShuOsher:1988} is considered as time discretization
 together with DG discretizations of arbitrary accuracy in phase domain. Second order Runge-Kutta time discretization is not used due to
 their otherwise restrictive limitation on timestep  when the spatial accuracy is higher than second order \cite{ZS2004, ZS2006}.
 The analysis is based on Taylor expansion, energy analysis, some techniques for analyzing the semi-discrete DG methods of the VM system \cite{CGLM} and for analyzing the third order Runge-Kutta time discretization within the method of lines framework.
 To treat the nonlinearity due to the nonlinear coupling of the Vlasov and Maxwell parts, the polynomial degree $k$ is required to satisfy $k \geq \left \lceil \frac{d_x + 1}{2} \right \rceil$.
 In addition, a priori assumption is made for the $L^\infty$ error of
 the electric and magnetic fields. This assumption will be  proved later by mathematical induction.

The remaining of this paper is organized as follows.
In section \ref{sec:2},  the formulations of RKDG methods are presented for the VM system. We also introduce notations and review some standard approximation results and inverse inequalities in finite element methods.
In section \ref{sec:3}, error estimates are established for the RKDG methods. Here we start with the error equations and energy equations. Based on these equations, the errors from the Vlasov and Maxwell parts are estimated and then combined.
To better present the analysis, the proofs of some lemmas are given  in section \ref{sec:4}.
Finally, we summarize and generalize our work in section \ref{sec:5}.

\section{Formulation of Runge-Kutta Discontinuous Galerkin Methods}
\label{sec:2}

In this section, we will introduce notations, review some standard approximation results and inverse inequalities,  and present the formulation of Runge-Kutta discontinuous Galerkin methods for the Vlasov-Maxwell system \eqref{eq1:vm1}.

  \subsection{Some preliminaries}

  Throughout this paper, standard notations are used for Sobolev spaces and norms: for a bounded domain $D$ and any nonnegative integer $m$, we denote the $L^2$-Sobolev space of order $m$ by $H^m(D)$
  equipped with Sobolev norm $||\cdot||_{m,D}$, and the $L^{\infty}$-Sobolev space of order $m$ by $W^{m,\infty}(D)$ with the Sobolev norm $||\cdot||_{m,\infty,D}$.
  When $m=0$, $L^2(D)$ is used  instead of $H^{0}(D)$, so is $L^{\infty}(D)$ instead of $W^{0,\infty}(D)$.
  For the brevity of notation, we use $\star=x$ or $v$ in this subsection.  For the computational domain $\Omega=\Omega_x\times\Omega_v$,
  assume $T_h^\star=\{K_\star\}$ is  a partition of $\Omega_\star$, with $K_\star$ being a
(rotated) Cartesian element or a  simplex,
then $T_h=\{K: K=K_x \times K_v, \forall K_x \in T_h^x, \forall K_v \in T_h^v\}$ defines a partition of $\Omega$.
  Let $\mathcal{E}_\star$ be the set of the edges of $T_h^\star$,
  then the edges of $T_h$ will be $\mathcal{E} = \{T_h^x \times \mathcal{E}_v \} \cup \{T_h^v \times \mathcal{E}_x \}$.
  In addition, let $\mathcal{E}_v = \mathcal{E}_v^i \cup \mathcal{E}_v^b$ with $\mathcal{E}_v^i$ (resp. $\mathcal{E}_v^b$) consisting of all interior (resp. boundary) edges of $T_h^v$.
  The mesh size of $T_h$ is denoted as $h=\max(h_x,h_v)=\max_{K\in T_h}h_{K}$, where $h_\star=\max_{K_\star \in T_h^\star} h_{K_\star}$ with $h_{K_\star} = $ diam($K_\star$),
  and  $h_K=\max(h_{K_x}, h_{K_v})$ with $K={K_x}\times{K_v}$.
  When the mesh is refined, we assume both $\frac{h_x}{h_{v,\min}}:=\frac{h_x}{\min_{K_v\in T_h^v}h_{K_v}}$
  and $\frac{h_v}{h_{x,\min}}:=\frac{h_v}{\min_{K_x\in T_h^x}h_{K_x}}$ are uniformly bounded above by a positive constant $\sigma_0$.
Therefore in our analysis we  do not always distinguish $h$, $h_{x}$, $h_{v}$, $h_{K_x}$ and $h_{K_v}$.
  It is further assumed that $\{T_h^\star\}_h$ is shape-regular. That is, if $\rho_{K_\star}$ denotes the diameter of the largest sphere included in $K_\star$, there is
$\frac{h_{K_\star}}{\rho_{K_\star}}\leq \sigma_\star,\; \forall K_\star\in\mT_h^\star$,
for a positive constant $\sigma_\star$ independent of $h_\star$.

  Next we introduce two finite dimensional discrete spaces
\begin{eqnarray*}
 \mathcal{G}_h^{k} &=& \{g \in L^2(\Omega): g|_{K} \in P^k(K),\forall K \in T_h \}, \\
 \mathcal{U}_h^{k} &=& \{U \in [L^2(\Omega_x)]^{3}: U|_{K_x} \in [P^k(K_x)]^{3}, \forall K_x \in T_h^x\},
\end{eqnarray*}
where $P^k(D)$ is the set of polynomials of the {\em total} degree at most $k$ on $D$, with $k$ being any nonnegative integer. Note that functions in $\mG_h^k$ (resp. $\mU_h^k$) are piecewise defined with respect to $T_h$ (resp. $T_h^x$). For such function, we would need the notations of jumps and averages.
Given an edge $e=(K_{x}^{+} \cap K_{x}^{-})\in\mE_x$ with $n_x^{\pm}$ as the outward unit normal vector of $K_x^\pm$, for any
$g \in \mathcal{G}_h^k$ and $U\in \mathcal{U}_h^k$ with
$g^{\pm} = g|_{K_{x}^{\pm}}$ and $U^{\pm} = U|_{K_{x}^{\pm}}$, the averages of $f$ and $U$ across $e$ are
\begin{equation*}
 \{g\}_x = \frac{1}{2}(g^{+} + g^{-}), \quad \{U\}_x = \frac{1}{2}(U^{+} + U^{-}),
\end{equation*}
and the jumps are
\begin{equation*}
 [g]_x=g^{+} n_{x}^{+} + g^{-} n_{x}^{-}, \quad [U]_x=U^{+} \cdot n_{x}^{+} + U^{-} \cdot n_{x}^{-}, \quad [U]_{tan}=U^{+}\times n_x^{+}+U^{-} \times n_x^{-}.
\end{equation*}
The averages and jumps across any interior edge $e=({K_{v}^{+} \cap K_{v}^{-} })\in\mE_v^i$ can be defined similarly.
For a boundary edge in $\mE_v^b$ with $n_v$ being the outward unit normal vector, we set  $[g]_v = g n_v$ and $\{g\}_v = \frac{1}{2}g$.
This is consistent with the exact solution $f$ being compactly supported in $\Omega_v$.
Below are some equalities which will be frequently used in the analysis and can be easily verified,
 \begin{subequations}
 \begin{align}
  \label{eq:jump1}
  &\frac{1}{2}[g^2]_{\star} = \{g\}_{\star} [g]_{\star},   \\
  \label{eq:jump3}
  &[g_1 g_2]_x - \{g_1\}_x [g_2]_x - \{g_2\}_x [g_1]_x = 0,  \\
  \label{eq:jump2}
  &[U \times V]_x + \{V\}_x \cdot [U]_{tan} - \{U\}_x \cdot [V]_{tan} = 0,
\end{align}
\end{subequations}
where
$g, g_1, g_2 \in \mathcal{G}_h^k$, and $U, V\in \mU_h^k$.
We also introduce some shorthand notations,
$\int_{\Omega} = \int_{T_h} = \sum_{K \in T_h} \int_K$, $\int_{\Omega_{\star}} = \int_{T_h^{\star}} = \sum_{K_{\star} \in T_h^{\star}} \int_{K_{\star}}$, $\int_{\mathcal{E}_{\star}} = \sum_{e \in \mathcal{E}_{\star}} \int_{e}$.
Additionally, $||g||_{0,\mathcal{E}} = (||g||_{0,\mathcal{E}_x \times T_h^v}^2 + ||g||_{0,T_h^x \times \mathcal{E}_v}^2)^{\frac{1}{2}}$ with
$||g||_{0,\mathcal{E}_x \times T_h^v} = \left( \int_{\mathcal{E}_x} \int_{T_h^v} g^2 dv ds_x \right)^{\frac{1}{2}}$,
$||g||_{0,T_h^x \times \mathcal{E}_v} = \left( \int_{T_h^x} \int_{\mathcal{E}_v} g^2 ds_v dx \right)^{\frac{1}{2}}$,
and $||g||_{0,\mathcal{E}_x} = (\int_{\mathcal{E}_x} g^2 ds_x)^{\frac{1}{2}}$.

Let $\Pi^k$ denote the $L^2$ projection onto $\mG_h^k$, and $\Pi^k_x$ be the $L^2$ projection onto $\mU_h^k$. In this paper, the following approximation properties in \eqref{eq:approx} and inverse inequalities in \eqref{eq:inv} will  be used:
there exists a constant $C>0$, such that $\forall g\in H^{k+1}(\Omega)$, $\forall U\in [H^{k+1}(\Omega_x)]^3$,
\begin{equation}
\label{eq:approx}
\left\{
\begin{aligned}
 ||g-\Pi^k g||_{0,K} + h_K^{\frac{1}{2}}||g-\Pi^k g||_{0,\partial K} \leq Ch_K^{k+1}||g||_{k+1,K}, \quad \forall K\in T_h, \\
 ||U-\Pi^k_x U||_{0,K_x} + h_{K_x}^{\frac{1}{2}}||U-\Pi^k_x U||_{0,\partial K_x} \leq Ch_{K_x}^{k+1}||U||_{k+1,K_x}, \quad \forall K_x \in T_h^x, \\
 ||U-\Pi^k_x U||_{0,\infty,K_x} \leq Ch_{K_x}^{k+1}||U||_{k+1,\infty,K_x}, \forall K_x \in T_h^x.\\
 \end{aligned}
 \right.
\end{equation}
In addition, there exists a constant $C>0$, such that $\forall g\in P^k(K)$, $\forall U\in [P^k(K_x)]^3$,
\begin{equation}
\label{eq:inv}
\left\{
\begin{aligned}
 ||\nabla_x g||_{0,K} \leq Ch_{K_x}^{-1}||g||_{0,K}, \quad ||\nabla_v g||_{0,K}\leq Ch_{K_v}^{-1}||g||_{0,K}, \\
 ||U||_{0,\infty,K_x} \leq Ch_{K_x}^{-\frac{d_x}{2}}||U||_{0,K_x}, \quad  ||U||_{0,\partial K_x} \leq Ch_{K_x}^{-\frac{1}{2}}||U||_{0,K_x}.
 \end{aligned}
 \right.
\end{equation}
Each positive constant $C$ in \eqref{eq:approx} and \eqref{eq:inv} is independent of the mesh sizes $h_\Kx$ and $h_{K_v}$, and it depends on $k$ and the shape regularity parameters $\sigma_x$ and (or) $\sigma_v$ of the mesh.  One can refer to \cite{Ciarlet} for more details of such standard results.

Throughout the paper,  $\tau$  is used to denote the time step and $t^n=n\tau$. Without loss of generality, we assume the time steps are uniform and $\tau$, $h \leq 1$.
The analysis in this paper also holds for non-uniform time steps.  Even though the numerical methods and error analysis will be presented when both the (exact and approximated) electric and magnetic fields have three components, they can be easily adapted to reduced VM equations, such as the one to study Weibel instability in \cite{CGLM} when $d_x=1$ and $d_v=2$,  where some components of $E$ and $B$ vanish and need not be approximated numerically.

\subsection{Runge-Kutta discontinuous Galerkin methods}
\label{sec:2.2}

Now we are ready to present the RKDG methods for the VM system, where upwind DG methods of arbitrary accuracy are used
as the spatial and phase discretization and a third order TVD Runge-Kutta method \cite{ShuOsher:1988} is used as
the time discretization.  Note that on the PDE level, the two equations in \eqref{eq1:vm1} involving the divergence
of the magnetic and electric fields can be derived from the remaining equations of the VM system as long as they
are satisfied by the initial data, these equations will not be discretized numerically just as in \cite{CGLM}.
For sufficiently smooth solutions as considered in this work, the divergence equations can be approximated accurately by the proposed methods (e.g.
with the accuracy which is one order lower than that of the electric and magnetic fields). One should be aware that for general cases,
imposing divergence equations in numerical simulations can be important.
To initialize the simulation, let  $f_h^0=\Pi^k f_0$, $E_h^0 = \Pi^k_{x} E_0$ and $B_h^0 =\Pi^k_{x} B_0$, where $f_0$, $E_0$ and $B_0$ are the initial data of the VM system.
 Then for $n \geq 0$, the approximate solutions at time $t^{n+1} = (n+1)\tau$ are defined as follows. We look for $f_h^{n,1}, f_h^{n,2}, f_h^{n+1} \in \mathcal{G}_h^k$, and $E_h^{n,1}, E_h^{n,2}, E_h^{n+1}, B_h^{n,1}, B_h^{n,2}, B_h^{n+1}\in \mathcal{U}_h^k$ satisfying
 \begin{subequations}
 \label{eq:1:2:3}
\begin{align}
   &(f_h^{n,1}, g)_\OO = (f_h^n, g)_\OO + \tau a_h(f_h^n, E_h^n, B_h^n;g),\label{eq:1} \\
   &(E_h^{n,1}, U)_\Ox + (B_h^{n,1}, V)_\Ox= (E_h^n, U)_\Ox + (B_h^n, V)_\Ox+\tau b_h(E_h^n, B_h^n, f_h^n; U, V),\notag\\
      &(f_h^{n,2}, g)_\OO = (\frac{3}{4}f_h^n+\frac{1}{4}f_h^{n,1}, g)_\OO + \frac{\tau}{4} a_h(f_h^{n,1}, E_h^{n,1}, B_h^{n,1};g), \label{eq:2}\\
   &(E_h^{n,2}, U)_\Ox + (B_h^{n,2}, V)_\Ox= (\frac{3}{4}E_h^n+\frac{1}{4}E_h^{n,1}, U)_\Ox + (\frac{3}{4}B_h^n+\frac{1}{4}B_h^{n,1}, V)_\Ox+\frac{\tau}{4} b_h(E_h^{n,1}, B_h^{n,1}, f_h^{n,1}; U, V),\notag\\
 &(f_h^{n+1}, g)_\OO = (\frac{1}{3}f_h^n+\frac{2}{3}f_h^{n,2}, g)_\OO + \frac{2\tau}{3} a_h(f_h^{n,2}, E_h^{n,2}, B_h^{n,2};g), \label{eq:3}\\
   &(E_h^{n+1}, U)_\Ox + (B_h^{n+1}, V)_\Ox= (\frac{1}{3}E_h^n+\frac{2}{3}E_h^{n,2}, U)_\Ox + (\frac{1}{3}B_h^n+\frac{2}{3}B_h^{n,2}, V)_\Ox+\frac{2\tau}{3} b_h(E_h^{n,2}, B_h^{n,2}, f_h^{n,2}; U, V),\notag
\end{align}
\end{subequations}
for any $g \in \mathcal{G}_h^k$ and $U,V \in \mathcal{U}_h^k$, where

\begin{align*}
 a_h(f_h,E_h,B_h;g)=&\int_\OO  f_h v\cdot \nabla_x g +  f_h(E_h+v\times B_h)\cdot \nabla_v g dxdv \\
 &- \sum_{K=K_x\times K_v\in T_h}\left(\int_{K_v}\int_{\partial K_x}\widehat{f_h v\cdot n_x}g ds_x dv + \int_{K_x}\int_{\partial K_v}\widehat{f_h(E_h+v\times B_h)\cdot n_v}g ds_v dx\right), \\
 b_h(E_h, B_h, f_h; U, V) =& \int_\Ox  B_h \cdot \nabla\times U - E_h \cdot \nabla\times V  dx
 + \sum_{K_x\in T_h^x}\int_{\partial K_x}\left(\widehat{n_x \times B_h}\cdot U - \widehat{n_x \times E_h}\cdot V \right) ds_x\\
 &-\int_\Ox J_h \cdot U dx, \quad J_h(x,t)=\int_\Ov f_h(x,v,t)v dv.
\end{align*}
Here $n_x$ and $n_v$ are outward unit normal vectors of $\partial K_x$ and $\partial K_v$, respectively. All the hat functions are upwinding numerical fluxes defined as
\begin{eqnarray*}
 \widehat{f_hv\cdot n_x}&=&
 \left(\{f_hv\}_x + \frac{|v\cdot n_x|}{2}[f_h]_x \right)\cdot n_x, \\
 \widehat{f_h(E_h+v\times B_h)\cdot n_v}
&=& \left(\{f_h(E_h + v\times B_h)\}_v + \frac{|(E_h+v\times B_h)\cdot n_v|}{2}[f_h]_v \right)\cdot n_v, \\
 \widehat{n_x \times E_h}&=&
 n_x \times \left(\{E_h\}_x + \frac{1}{2}[B_h]_{tan} \right), \quad
 \widehat{n_x \times B_h}=
 n_x \times \left(\{B_h\}_x - \frac{1}{2}[E_h]_{tan} \right),
\end{eqnarray*}
and they further specify $a_h(f_h,E_h,B_h;g)=a_{h,1}(f_h;g)+
a_{h,2}(f_h,E_h,B_h;g)$ with
\begin{align*}
 a_{h,1}(f_h;g) =& \int_{\Omega} f_h v\cdot \nabla_x g dxdv - \int_{T_h^v}\int_{\mathcal{E}_x}\left(\{f_h v\}_x + \frac{|v\cdot n_x|}{2}[f_h]_x \right)\cdot [g]_x ds_x dv, \\
 a_{h,2}(f_h,E_h,B_h;g) =& \int_{\Omega}f_h(E_h+ v\times B_h)\cdot \nabla_v g dxdv \\
 &- \int_{T_h^x}\int_{\mathcal{E}_v}\left(\{f_h(E_h + v\times B_h) \}_v + \frac{|(E_h + v\times B_h)\cdot n_v|}{2} [f_h]_v \right)\cdot[g]_v ds_v dx,
 \end{align*}
 and
 \begin{align*}
  b_h(E_h, B_h, f_h; U, V)
  =& \int_{\Omega_x} \left( B_h \cdot \nabla \times U - E_h \cdot \nabla \times V \right) dx-\int_\Ox J_h U dx \\
  &+ \int_{\mathcal{E}_x} \left( \{B_h\}_x - \frac{1}{2}[E_h]_{tan} \right) \cdot [U]_{tan}-\left( \{E_h\}_x + \frac{1}{2}[B_h]_{tan} \right) \cdot [V]_{tan} ds_x.
  \end{align*}

 Note that both $a_{h,1} $ and $b_h$ are linear with respect to each argument, yet $a_{h,2}(f_h, E_h, B_h;g)$ is linear with respect to $f_h$ and $g$ only.
 The overall RKDG methods are consistent, and this will be used to derive the error equations in next section.

\section{Error Estimates}
\label{sec:3}

This section is devoted to the main result of the paper, which is given in Theorem \ref{theorem:1}. More specifically, we will establish error estimates
at any given time $T>0$ for the fully discrete RKDG methods in section \ref{sec:2.2} when they are used to solve sufficiently smooth solutions.

Unless otherwise specified, $C$ is used to denote a generic positive constant, and it can take different values at different occurrences.
This constant is independent of $n, h, \tau$, and may depend on  polynomial degree $k$, mesh parameter $\sigma_0, \sigma_x, \sigma_v$, domain parameters $L_{x,i}$, $L_{v,i}$, $i=1,2$, and the time $T$.
It may also depend on  the exact solution in the form of its certain Sobolev norms or semi-norms.
The constant $\gamma_1$ in Theorem \ref{thm:vlas}, $\gamma_2$ in Theorem \ref{thm:max}, and $\gamma$ in Theorem \ref{theorem:1} have similar dependence as the generic constant $C$.
For convenience, we do not distinguish the upper indices $n,0$ and $n$. For instance, we regard $g^{n,0} = g^{n}$ for any function $g$.  We use $\lceil a \rceil$ to denote the smallest integer greater than or equal to $a$.
With the analysis being very technical, to make it easier to follow, the proofs of some lemmas are given later in section \ref{sec:4}.

\begin{theorem}
  \label{theorem:1}
 Let $(f, E, B)$ be a sufficiently smooth exact solution to the VM system \eqref{eq1:vm1}. Let $(f_h^n, E_h^n, B_h^n)\in \mG_h^k\times \mU_h^k\times \mU_h^k$ be the solution to the scheme \eqref{eq:1:2:3} at time $t^n$
  with   $k \geq \left \lceil \frac{d_x+1}{2} \right \rceil$.
 Under a CFL condition $\tau \leq \gamma h$, there is
\begin{equation*}
||f(\cdot, \cdot, t^n)-f_h^n||_{0,\Omega}^2 + ||E(\cdot, t^n)-E_h^n||_{0,\Omega_x}^2 + ||B(\cdot, t^n)-B_h^n||_{0,\Omega_x}^2 \leq Ch^{2k+1} + C\tau^6
\end{equation*}
 for any $n\leq T/\tau $. In addition, $\forall m+1 \leq T/\tau$,
 there is
 \begin{equation*}
||E(\cdot, t^m) - E_h^m||_{0,\infty,\Omega_x} \leq Ch, \quad ||B(\cdot, t^m) - B_h^m||_{0,\infty,\Omega_x} \leq Ch.
 \end{equation*}

\end{theorem}

\subsection{Error equations}
  Let $(f(x,v,t^n),E(x,t^n),B(x,t^n))$ be the exact solution to the system $\eqref{eq1:vm1}$ and $\eqref{eq2:vm2}$ at time $t^n = n\tau$.
 We denote $f^{n}(x,v)=f(x,v,t^n)$, $E^{n}(x)=E(x,t^n)$, $B^{n}(x)=B(x,t^n)$, $J^{n}(x) = \int_{\Omega_v}f^{n}(x,v)v dv$, and define

\begin{equation}
 \label{eq:exact}
\left\{
\begin{aligned}
 &f^{n,1} = f^{n} - \tau \left(v\cdot \nabla_x f^{n}+(E^{n}+v\times B^{n})\cdot \nabla_v f^{n} \right), \\
 &E^{n,1} = E^{n} + \tau (\nabla \times B^{n} - J^{n}), \quad B^{n,1} = B^{n} - \tau (\nabla \times E^{n}),\\
 &f^{n,2} = \frac{3}{4}f^{n}+\frac{1}{4}f^{n,1}-\frac{\tau}{4} \left(v\cdot \nabla_x f^{n,1} + (E^{n,1} + v\times B^{n,1} )\cdot \nabla_v f^{n,1}\right),\\
 &J^{n,1} = \int_{\Omega_v}f^{n,1} v dv, \quad J^{n,2} = \int_{\Omega_v}f^{n,2} v dv, \\
 &E^{n,2} = \frac{3}{4}E^{n}+\frac{1}{4}E^{n,1}+\frac{\tau}{4}(\nabla \times B^{n,1} - J^{n,1}),
 \quad B^{n,2} = \frac{3}{4} B^{n} + \frac{1}{4} B^{n,1} -\frac{\tau}{4}(\nabla \times E^{n,1}).
\end{aligned}
\right.
\end{equation}
Equations in $\eqref{eq:exact}$ are obtained by applying one step of the third order Runge-Kutta time discretization in \cite{ShuOsher:1988} to the VM system (without the divergence conditions) from $t=t^n$ with the exact solution as the initial data at $t_n$.
From $\eqref{eq1:vm1}$ and $\eqref{eq:exact}$,  one can further represent $f^{n,\sharp}$, $E^{n,\sharp}$ and $B^{n,\sharp}$ $(\sharp = 1,2)$ in terms of $f^{n}$, $E^{n}$, $B^{n}$ and their derivatives as below,

\begin{equation}
 \label{eq:exact_new}
\left\{
\begin{aligned}
 &f^{n,1} = f^{n} + \tau \partial_t f^{n},   \quad 
  E^{n,1} = E^{n} + \tau \partial_t E^{n},  \quad  
  B^{n,1} = B^{n} + \tau \partial_t B^{n}, \\
 &f^{n,2} = f^{n} + \frac{\tau}{2} \partial_t f^{n} + \frac{\tau^2}{4} \partial_t^2 f^{n} - \frac{\tau^3}{4} (\partial_t E^{n} + v\times \partial_t B^{n}) \cdot \nabla_v (\partial_t f^{n}),\\
 &E^{n,2} = E^{n} + \frac{\tau}{2} \partial_t E^{n} + \frac{\tau^2}{4} \partial_t^2 E^{n},      \quad  
  B^{n,2} = B^{n} + \frac{\tau}{2} \partial_t B^{n} + \frac{\tau^2}{4} \partial_t^2 B^{n}.
\end{aligned}
\right.
\end{equation}

In the next lemma,
local truncation errors from each step of the Runge-Kutta time discretization are given. The results can be verified straightforwardly based on \eqref{eq:exact_new} and Taylor expansion, and the proof is omitted.
\begin{lemma}
\label{lemma:1}
If we define
 \begin{eqnarray}
  f(x,v,t^{n+1})&=&\frac{1}{3}f^{n}+\frac{2}{3}f^{n,2}-\frac{2\tau}{3}\left( v\cdot \nabla_x f^{n,2}+ (E^{n,2}+v\times B^{n,2})\cdot \nabla_v f^{n,2} \right)
   + T_f^n(x,v), \nonumber\\
  E(x,t^{n+1})&=&\frac{1}{3}E^{n}+\frac{2}{3}E^{n,2}+\frac{2\tau}{3} (\nabla\times B^{n,2} - J^{n,2}) + T_E^n(x),   \label{eq:err0}\\
  B(x,t^{n+1})&=& \frac{1}{3}B^{n} + \frac{2}{3}B^{n,2}- \frac{2\tau}{3}(\nabla \times E^{n,2}) + T_B^n(x), \nonumber
 \end{eqnarray}
where $T_f^n(x, v), T_B^n(x), T_E^n(x)$ are the local truncation errors in the $n$-th time step $\forall n: (n+1)\tau\leq T$, then
  \begin{equation*}
    ||T_f^n||_{0,\Omega}, ||T_E^n||_{0,\Omega_x}, ||T_B^n||_{0,\Omega_x} \leq C\tau^4.
  \end{equation*}
  \end{lemma}

For each stage of the Runge-Kutta method, we denote the error by $e_f^{n,\sharp}=f^{n,\sharp}-f_h^{n,\sharp}=\xi_f^{n,\sharp}-\eta_f^{n,\sharp}$, where $\xi_f^{n,\sharp}=\Pi^{k}f^{n,\sharp}-f_h^{n,\sharp}$ and $\eta_f^{n,\sharp}=\Pi^{k}f^{n,\sharp}-f^{n,\sharp}$, for $\sharp = 0,1,2$. Given that $\eta_f^{n,\sharp}$ can be estimated in a standard way from the approximation results in \eqref{eq:approx} of the discrete spaces, our error analysis will focus on the term $\xi_f^{n,\sharp}$, which is also called the projected error due to $\xi_f^{n,\sharp}=\Pi^{k} e_f^{n,\sharp}$.  Similar comments and conventions on notation go to  $e_E^{n,\sharp}, e_B^{n,\sharp}, \xi_E^{n,\sharp}, \xi_B^{n,\sharp}, \eta_E^{n,\sharp}$ and $\eta_B^{n,\sharp}$.
Next we multiply an arbitrary test function $g\in \mathcal{G}_h^k$ (resp. $U, V \in \mathcal{U}_h^k$) on both sides of the equations corresponding to the Vlasov equation (resp. Maxwell equations) in \eqref{eq:exact} and \eqref{eq:err0}, integrate over each mesh element $K$ (resp. $\Kx$), take integration by parts, and sum up with respect to all $K\in T_h$  (or all $\Kx\in T_h^x$). We then subtract \eqref{eq:1} - \eqref{eq:3} from the resulting equations, and reach the error equations,
\begin{eqnarray}
  \label{eq:err1}
  \left\{
  \begin{aligned}
   &(\xi_f^{n,1},g)_{\Omega} = (\xi_f^n,g)_{\Omega} + \tau \mathcal{J}(g), \\
   &(\xi_f^{n,2},g)_{\Omega} = (\frac{3}{4}\xi_f^{n} + \frac{1}{4}\xi_f^{n,1},g)_{\Omega} + \frac{\tau}{4} \mathcal{K}(g), \\
   &(\xi_f^{n+1},g)_{\Omega} = (\frac{1}{3} \xi_f^{n} + \frac{2}{3}\xi_f^{n,2},g)_{\Omega} + \frac{2\tau}{3} \mathcal{L}(g),
  \end{aligned}
  \right.
\end{eqnarray}

\begin{eqnarray}
  \label{eq:err2}
  \left\{
  \begin{aligned}
   &(\xi_E^{n,1},U)_{\Omega_x}+ (\xi_B^{n,1},V)_{\Omega_x}= (\xi_E^n,U)_{\Omega_x}+(\xi_B^n,V)_{\Omega_x} + \tau \mathcal{Q}(U, V), \\
   &(\xi_E^{n,2},U)_{\Omega_x}+(\xi_B^{n,2},V)_{\Omega_x}= (\frac{3}{4}\xi_E^{n} + \frac{1}{4}\xi_E^{n,1},U)_{\Omega_x} +
   (\frac{3}{4} \xi_B^{n}+\frac{1}{4}\xi_B^{n,1},V)_{\Omega_x} +\frac{\tau}{4}  \mathcal{R}(U, V), \\
   &(\xi_E^{n+1},U)_{\Omega_x}+(\xi_B^{n+1},V)_{\Omega_x} = (\frac{1}{3} \xi_E^{n} + \frac{2}{3}\xi_E^{n,2},U)_{\Omega_x} + (\frac{1}{3} \xi_B^{n} + \frac{2}{3}\xi_B^{n,2},V)_{\Omega_x}+ \frac{2\tau}{3}\mathcal{S}(U, V).
  \end{aligned}
  \right.
\end{eqnarray}
  Here
  \begin{eqnarray*}
  \mathcal{J}(g)&=&\left(\frac{\eta_f^{n,1}-\eta_f^{n}}{\tau},g\right)_{\Omega} + a_h(f^n,E^n,B^n;g)-a_h(f_h^n,E_h^n,B_h^n;g),\\
  \mathcal{K}(g)&=&\left(\frac{4\eta_f^{n,2}-3\eta_f^n-\eta_f^{n,1}}{\tau},g\right)_{\Omega} + a_h(f^{n,1},E^{n,1},B^{n,1};g)-a_h(f_h^{n,1},E_h^{n,1},B_h^{n,1};g),\\
  \mathcal{L}(g)&=&\left(\frac{3\eta_f^{n+1}-\eta_f^n-2\eta_f^{n,2}+3T_f^n(x,v)}{2\tau}  ,g\right)_{\Omega}
      + a_h(f^{n,2},E^{n,2},B^{n,2};g)- a_h(f_h^{n,2},E_h^{n,2},B_h^{n,2};g),
  \end{eqnarray*}
      \begin{eqnarray*}
  \mathcal{Q}(U, V)&=&\left(\frac{\eta_E^{n,1}-\eta_E^{n}}{\tau},U\right)_{\Omega_x} + \left(\frac{\eta_B^{n,1}-\eta_B^{n}}{\tau},V\right)_{\Omega_x}
  +b_h(e_E^n, e_B^n, e_f^n; U, V),\\
  \mathcal{R}(U, V)&=&\left(\frac{4\eta_E^{n,2}-3\eta_E^n-\eta_E^{n,1}}{\tau},U\right)_{\Omega_x} +
  \left(\frac{4\eta_B^{n,2}-3\eta_B^n-\eta_B^{n,1}}{\tau},V\right)_{\Omega_x}
  + b_h(e_E^{n,1}, e_B^{n,1}, e_f^{n,1}; U, V),\\
  \mathcal{S}(U, V)&=&\left(\frac{3\eta_E^{n+1}-\eta_E^n-2\eta_E^{n,2}+3T_E^n(x)}{2\tau}  ,U\right)_{\Omega_x}
  + \left(\frac{3\eta_B^{n+1}-\eta_B^n-2\eta_B^{n,2}+3T_B^n(x)}{2\tau}  ,V\right)_{\Omega_x}\\
      &&+ b_h(e_E^{n,2}, e_B^{n,2}, e_f^{n,2}; U, V),
  \end{eqnarray*}
with any test functions $g\in \mathcal{G}_h^k$ and $U,V\in \mathcal{U}_h^k$. For the functional $\mathcal{J}(\cdot)$, we denote $\mathcal{J}_1(g)=\left(\frac{\eta_f^{n,1}-\eta_f^{n}}{\tau} ,g\right)_\Omega$ and $\mathcal{J}_2(g)=a_h(f^n,E^n,B^n;g)-a_h(f_h^n,E_h^n,B_h^n;g)$. Similarly, one can define $\mathcal{K}_{\sharp}$, $\mathcal{L}_{\sharp}$, $\mathcal{Q}_{\sharp}$, $\mathcal{R}_{\sharp}$, $\mathcal{S}_{\sharp}$, $\sharp=1,2$.

We now take the test function $g=\xi_f^n, 4\xi_f^{n,1}$ and $6\xi_f^{n,2}$ in each equation of \eqref{eq:err1},
respectively, sum them up and obtain the energy equations
\begin{eqnarray}
 \label{eq:errf}
 3||\xi_f^{n+1}||_{0,\Omega}^2 - 3||\xi_f^{n}||_{0,\Omega}^2
 &=& \tau[\mathcal{J}(\xi_f^n) + \mathcal{K}(\xi_f^{n,1}) + 4\mathcal{L}(\xi_f^{n,2})] \\
 &+& ||2\xi_f^{n,2} - \xi_f^{n,1} - \xi_f^{n}||_{0,\Omega}^2 + 3(\xi_f^{n+1}-\xi_f^n, \xi_f^{n+1}-2\xi_f^{n,2} + \xi_f^{n})_{\Omega}. \nonumber
\end{eqnarray}
Similarly, the following equation holds
\begin{eqnarray}
 \label{eq:errEB}
 &&3(||\xi_E^{n+1}||_{0,\Omega_x}^2 +  ||\xi_B^{n+1}||_{0,\Omega_x}^2) - 3(||\xi_E^{n}||_{0,\Omega_x}^2 + ||\xi_B^{n}||_{0,\Omega_x}^2)\\
 &=&
 \tau[\mathcal{Q}(\xi_E^n, \xi_B^n) + \mathcal{R}(\xi_E^{n,1}, \xi_B^{n,1}) + 4\mathcal{S}(\xi_E^{n,2}, \xi_B^{n,2})]+ ||2\xi_E^{n,2} - \xi_E^{n,1} - \xi_E^{n}||_{0,\Omega_x}^2 + ||2\xi_B^{n,2} - \xi_B^{n,1} - \xi_B^{n}||_{0,\Omega_x}^2\notag\\
 &&+3(\xi_E^{n+1}-\xi_E^n, \xi_E^{n+1}-2\xi_E^{n,2} + \xi_E^{n})_{\Omega_x} + 3(\xi_B^{n+1}-\xi_B^n, \xi_B^{n+1}-2\xi_B^{n,2} + \xi_B^{n})_{\Omega_x}. \nonumber
\end{eqnarray}
The main error estimate will be established based on equations \eqref{eq:errf} - \eqref{eq:errEB} which describe how the $L^2$ norms of the projected errors are accumulated in one time step. In particular, in the next two subsections, we will estimate the errors from the Vlasov and Maxwell solvers, respectively, and the results will be combined in section \ref{sec:3.4} to get the main result of this paper.

Before continue, we will make {\em a priori} assumption
for the $L^\infty$ error of the magnetic and electric fields,

\bigskip
\noindent{\bf $L^\infty$-Assumption}: For any integer $n+1 \leq T/\tau$,
there is $||e_E^{n,\sharp}||_{0,\infty,\Omega_x}$, $||e_B^{n,\sharp}||_{0,\infty,\Omega_x} \leq Ch$ with $\sharp = 0,1,2$.

\bigskip
\noindent
This assumption will be used in section \ref{sec:3.2} and Lemma \ref{lemma:2}-(2) to estimate terms with $a_{h,2}$ as this is where the nonlinear coupling of the Vlasov and Maxwell parts lies, and this assumption will eventually be established rigorously by mathematical induction in section \ref{sec:3.4}.

Our analysis will also benefit from the following shorthand notations,
\begin{subequations}
\begin{align}
\stab_f^\star&=\int_{T_h^v} \int_{\mathcal{E}_x} |v\cdot n_x| |[\xi_f^\star]_x|^2 ds_x dv
 + \int_{T_h^x} \int_{\mathcal{E}_v} |(E_h^\star + v \times B_h^\star)\cdot n_v| |[\xi_f^\star]_v|^2 ds_v dx,\\
 \stab_{EB}^\star&=\int_{\mathcal{E}_x}( |[\xi_E^{\star}]_{tan}|^2 + |[\xi_B^{\star}]_{tan}|^2 )ds_x,
 \quad \ene^\star =||\xi_f^\star||^2_{0,\Omega} + ||\xi_E^{\star}||_{0,\Omega_x}^2 + ||\xi_B^\star||_{0,\Omega_x}^2.
\end{align}
\end{subequations}
where $\star=n$ or $n,\sharp$.
 The terms $\stab$ with different subscripts or superscripts provide stability mechanism due to the upwind phase and spatial discretizations. Later on another type of stability mechanism will emerge which is due to the temporal discretization.

In our analysis, we will frequently encounter certain linear combinations of $\eta_\diamond^{n}$, $\eta_\diamond^{n,1}$, $\eta_\diamond^{n,2}$ and $\eta_\diamond^{n+1}$, $\diamond=f, E, B$. In the next lemma, the estimates for such terms are summarized, with their proofs given in section \ref{sec:4}.

\begin{lemma}
\label{lemma:2}
 Let $d^n_\diamond = d_0 \eta_\diamond^n + d_1 \eta_\diamond^{n,1} + d_2 \eta_\diamond^{n,2} + d_3 \eta_\diamond^{n+1}$, $\diamond=f, E, B$, where $d_0, d_1, d_2, d_3$ are four constants satisfying $d_0+d_1 + d_2+d_3=0$ and independent of $n,h,\tau$. Then for any $g \in \mG_h^k$ and $U, V \in \mU_h^k$, we have
\begin{subequations}
 \begin{align}
(1)&\;\;   ||d^n_f||_{0,\Omega} + h^{\frac{1}{2}} ||d^n_f||_{0,\mathcal{E}} \leq C \tau h^{k+1}, \quad ||d^n_\star||_{0,\Omega_x} + h_x^{\frac{1}{2}} ||d^n_\star||_{0,\mathcal{E}_x} \leq C\tau h_x^{k+1}, \;\star=E, B, \label{eq:2.2.1}\\
(2)&\;\; |a_h(d^n_f, E_h^{n,s}, B_h^{n,s}; g)|
\leq   C  \frac{\tau}{h} h^{k+1} ||g||_{0,\Omega}\leq  C  \frac{\tau}{h} (h^{2k+2} + ||g||_{0,\Omega}^2), \;s=0, 1, 2, \label{eq:2.2.1_1}\\
(3)&\;\; |b_h(d_E^n, d_B^n, d_f^n; U, V)| \leq
        C \tau h^{k} (||U||_{0,\Omega_x}^2 + ||V||_{0,\Omega_x}^2).\label{eq:2.2.1_4}
  \end{align}
\end{subequations}
\end{lemma}

\subsection{The Vlasov equation part}
\label{sec:3.2}

We start with a key decomposition of the error change in one time step \cite{ZS2010}, namely, $\xi_f^{n+1} - \xi_f^{n} = G_1^n + G_2^n + G_3^n$, where $G_1^n = \xi_f^{n,1}-\xi_f^{n}, G_2^n = 2\xi_f^{n,2}-\xi_f^{n,1}-\xi_f^{n}$ and  $G_3^n = \xi_f^{n+1} - 2\xi_f^{n,2} + \xi_f^n$. It is obvious that
\begin{equation}
||G_i^n||_{0,\Omega}^2 \leq C\sum_{\sharp=0}^2 ||\xi_f^{n,\sharp}||_{0,\Omega}^2,\quad i=1, 2, 3.\label{eq:G=f}
\end{equation}
From equations \eqref{eq:err1}, one gets
\begin{subequations}
\begin{eqnarray}
  (G_2^n, g)_{\Omega} &=& \frac{\tau}{2}(\mathcal{K}(g)-\mathcal{J}(g))\equiv \frac{\tau}{2}\mathcal{K}_{RK}(g), \\
  (G_3^n, g)_{\Omega} &=& \frac{\tau}{3}(2\mathcal{L}(g)-\mathcal{K}(g)-\mathcal{J}(g))\equiv \frac{\tau}{3}\mathcal{L}_{RK}(g),
\end{eqnarray}
\end{subequations}
for any $g\in \mG_h^k$.
In addition, one can verify based on \eqref{eq:errf} that
\begin{equation}
 3||\xi_f^{n+1}||_{0,\Omega}^2 - 3||\xi_f^{n}||_{0,\Omega}^2 = \Xi_1+\Xi_2+\Xi_3, \label{eq:101}
\end{equation}
where $\Xi_i= \tau[\mathcal{J}_i(\xi_f^n) + \mathcal{K}_i(\xi_f^{n,1}) + 4\mathcal{L}_i(\xi_f^{n,2})]$, $i=1, 2$
and
$$\Xi_3= (G_2^n, G_2^n)_{\Omega} + 3(G_1^n, G_3^n)_{\Omega} + 3(G_2^n, G_3^n)_{\Omega} + 3(G_3^n, G_3^n)_{\Omega}.$$
In particular,
$\Xi_2$ relies on the phase space discretizations, for which some results were essentially established in the analysis of the semi-discrete DG methods for  the VM system in \cite{CGLM}.
On the other hand, $\Xi_1$ and $\Xi_3$ characterize more the contribution of the time discretization. One will see that there are two mechanisms contributing to numerical stability, one is $\stab_f^\star$ ($\star$ can be $n$ or $n, \sharp$) which comes from the phase space discretization and is also used in analyzing the semi-discrete method in \cite{CGLM}, the other one is $||G_2^n||^2$ which comes from the third order Runge-Kutta time discretization.

\bigskip
We first summarize in Lemma \ref{3.1} some estimates, which are based on the phase space discretization and are  essentially available in the analysis of the semi-discrete upwind DG method in \cite{CGLM}. For completeness, the proofs are given in section \ref{sec:4}.


\begin{lemma}
\label{3.1}
 For $\sharp = 0,1,2$, we have
 \begin{subequations}
 \begin{align}
 (1)& \;\;\ a_h(\xi_f^{n,\sharp}, E_h^{n,\sharp}, B_h^{n,\sharp}; \xi_f^{n,\sharp}) = -\frac{1}{2} \stab_f^{n,\sharp}, \label{eq:3.1.2}\\
 (2)&\;\;  a_h(\eta_f^{n,\sharp}, E_h^{n,\sharp}, B_h^{n,\sharp};\xi_f^{n,\sharp} ) \leq Ch^{2k+1} + C\ene^{n,\sharp}+\frac{1}{16}\stab_f^{n,\sharp}, \quad \textrm{for}\; k\geq \left \lceil \frac{d_x}{2} \right \rceil,\label{eq:3.1.2_1}\\
  (3)&\;\;  |a_h(f^{n,\sharp}, E^{n,\sharp}, B^{n,\sharp}; g) - a_h(f^{n,\sharp}, E^{n,\sharp}_h, B^{n,\sharp}_h; g)|
  \leq C (||e_E^{n,\sharp}||_{0,\Omega} + C||e_B^{n,\sharp}||_{0,\Omega}) ||g||_{0,\Omega},\quad \forall g\in \mG_h^k, \label{eq:3.1.2_2_0}\\
 & \textrm{moreover} \;\;|a_h(f^{n,\sharp}, E^{n,\sharp}, B^{n,\sharp}; \xi_f^{n,\sharp}) - a_h(f^{n,\sharp}, E^{n,\sharp}_h, B^{n,\sharp}_h; \xi_f^{n,\sharp})|
  \leq C h^{2k+2} + C\ene^{n,\sharp}. \label{eq:3.1.2_2}
\end{align}
\end{subequations}
\end{lemma}
\begin{proposition}
\label{3.4}
The following estimates hold for $\Xi_1$ and $\Xi_2$,
\begin{subequations}
 \begin{align}
 (1)&\;\;\; \Xi_1 \leq C\tau(h^{2k+2} +\tau^6)+  C\tau\sum_{\sharp=0}^2||\xi_f^{n,\sharp}||_{0,\Omega}^2,\label{eq:est:xi0}\\
 (2)&\;\;\; \Xi_2
  \leq  C\tau h^{2k+1} + C\tau \sum_{\sharp=0}^2 \ene^{n,\sharp}
  -\frac{7}{16}\tau \left(\stab_f^{n}+\stab_f^{n,1}+4\;\stab_f^{n,2}\right), \forall k \geq \left \lceil \frac{d_x}{2} \right \rceil.  \label{eq:est:xi1}
\end{align}
\end{subequations}
\end{proposition}

\begin{proof}
 Recall $\Xi_1= \tau[\mathcal{J}_1(\xi_f^n) + \mathcal{K}_1(\xi_f^{n,1}) + 4\mathcal{L}_1(\xi_f^{n,2})]$, with similarity, we only estimate
       $\mathcal{L}_1(\xi_f^{n,2})$. Applying Cauchy-Schwarz inequality, \eqref{eq:2.2.1} in Lemma \ref{lemma:2}, and truncation error estimate in Lemma \ref{lemma:1}, we get
  \begin{eqnarray}
  \label{eq:3.4.5}
  \mathcal{L}_1(\xi_f^{n,2}) &=& \left( \frac{3\eta_f^{n+1} -\eta_f^n - 2\eta_f^{n,2} + 3T_f^n(x,v)}{2\tau},\xi_f^{n,2} \right)_{\Omega} \nonumber\\
  &\leq& \frac{1}{2\tau} \left(||3\eta_f^{n+1} -\eta_f^n - 2\eta_f^{n,2}||_{0,\Omega} + 3||T_f^n||_{0,\Omega} \right)
  ||\xi_f^{n,2}||_{0,\Omega}                                    \nonumber\\
  &\leq& C(h^{k+1} + \tau^3) ||\xi_f^{n,2}||_{0,\Omega}
   \leq C(h^{2k+2} +\tau^6)+  C||\xi_f^{n,2}||_{0,\Omega}^2.
 \end{eqnarray}
 To estimate $\Xi_2$, due to similarity, we will only estimate $\mathcal{J}_2(\xi_f^n)$. Using the results in Lemma \ref{3.1}, one has
  \begin{eqnarray*}
  \mathcal{J}_2(\xi_f^n) &=& a_h(f^n, E^n, B^n; \xi_f^n) - a_h(f^n_h, E^n_h, B^n_h; \xi_f^n)       \notag\\
  &=& a_h(\xi_f^n, E^n_h, B^n_h; \xi_f^n) + (a_h(f^n, E^n, B^n; \xi_f^n) - a_h(f^n, E^n_h, B^n_h; \xi_f^n))- a_h(\eta_f^n, E^n_h, B_h^n; \xi_f^n) \notag\\
   &\leq& C h^{2k+1} + C\ene^n-\frac{7}{16}\stab_f^n.
   \end{eqnarray*}
\end{proof}

\bigskip
Next we will estimate $\Xi_3$. One key is to use $-||G_2^n||_{0,\Omega}^2$ to control $(C \frac{\tau}{h}+C\frac{\tau^2}{h^2})||G_2^n||_{0,\Omega}^2$
under some condition on the time step $\tau$.
 To make the details tractable, we first give some preparatory results.


\begin{lemma}
\label{lemma:3.5_7}
 For $r=0,1,2$, $s=0,1$, and any $g \in \mathcal{G}_h^k$
 \begin{subequations}
 \begin{eqnarray}
 (1)&&   |a_h(\eta_f^{n,r}, E_h^{n,s+1}, B_h^{n,s+1}; g) - a_h(\eta_f^{n,r}, E_h^{n,s}, B_h^{n,s}; g)|
  \leq  C(1 + \frac{\tau}{h}) h^{k+1} ||g||_{0,\Omega}, \label{eq:lem:3.5_7.1}\\
 (2) & & |a_h(\xi_f^{n,r}, E_h^{n,s+1}, B_h^{n,s+1}; g) - a_h(\xi_f^{n,r}, E_h^{n,s}, B_h^{n,s}; g)|
  \leq C(1 + \frac{\tau}{h}) ||\xi_f^{n,r}||_{0,\Omega} ||g||_{0,\Omega}.\label{eq:lem:3.5_7.2}
 \end{eqnarray}
 \end{subequations}
\end{lemma}

%
%
\begin{lemma}
\label{lemma:3.9}
 For any $g \in \mathcal{G}_h^k$, we have
 \begin{subequations}
  \begin{align}
   \mathcal{L}_{RK}(g)
 &\leq
 C\left((1+\frac{\tau}{h})\sum_{\sharp=0}^{2} (||\xi_E^{n,\sharp}||_{0,\Omega_x} + ||\xi_B^{n,\sharp}||_{0,\Omega_x} + ||\xi_f^{n,\sharp}||_{0,\Omega}+h^{k+1}) +\tau^3 \right) ||g||_{0,\Omega} +a_h(G_2^n, E_h^{n,1}, B_h^{n,1}; g) \label{eq:lem:3.9_1}\\
 &\leq
  C\left((1+\frac{\tau}{h})\sum_{\sharp=0}^{2} (||\xi_E^{n,\sharp}||_{0,\Omega_x} + ||\xi_B^{n,\sharp}||_{0,\Omega_x} + ||\xi_f^{n,\sharp}||_{0,\Omega}+ h^{k+1}) +\tau^3 + \frac{||G_2^n||_{0,\Omega}}{h} \right) ||g||_{0,\Omega}.\label{eq:lem:3.9_2}\\
  \mathcal{K}_{RK}(g)
 &\leq
 C(1+\frac{\tau}{h})\left(\sum_{\sharp=0}^{2} (||\xi_E^{n,\sharp}||_{0,\Omega_x} + ||\xi_B^{n,\sharp}||_{0,\Omega_x} + ||\xi_f^{n,\sharp}||_{0,\Omega}+h^{k+1}) \right) ||g||_{0,\Omega} +a_h(G_1^n, E_h^{n,1}, B_h^{n,1}; g). \label{eq:lem:3.9_3}
 \end{align}
 \end{subequations}
\end{lemma}


\begin{lemma}
\label{lemma:3.8}
 \begin{eqnarray}
 & &|a_h(G_1^n, E_h^{n,1}, B_h^{n,1}; G_2^n) + a_h(G_2^n, E_h^{n,1}, B_h^{n,1}; G_1^n)|                           \nonumber\\
 &\leq& C \left(1 + \frac{\tau}{h} \right) ||\xi_f^n||_{0,\Omega}^2 + \frac{C}{h}||G_2^n||_{0,\Omega}^2
   + \frac{1}{16} (\stab_f^{n} + \stab_f^{n,1}).
 \end{eqnarray}
\end{lemma}

With all the preparation in Lemmas \ref{lemma:3.5_7} - \ref{lemma:3.8}, we are now ready to estimate $\Xi_3$.

\begin{proposition}
\label{lemma:3.10}
\begin{eqnarray}
\label{eq:X4}
  \Xi_{3}
  &\leq&
  C \left( \tau (1+\frac{\tau}{h}) +
  \tau^2 (1+\frac{\tau}{h})^2 \right)
  ( \sum_{\sharp=0}^2 \ene^{n,\sharp} + h^{2k+2} )
  + C\tau^7 \nonumber\\
  &+&
  \left( -1 + C \frac{\tau}{h} + C \frac{\tau^2}{h^2} \right)
  ||G_2^n||_{0,\Omega}^2
  +
  \frac{\tau}{16} (\stab_f^{n} + \stab_f^{n,1}).
\end{eqnarray}
\end{proposition}
\begin{proof}
First note that
\begin{eqnarray}
\Xi_3 & =&  - ||G_2^n||_{0,\Omega}^2 + 2(G_2^n,G_2^n)_{\Omega}+ 3(G_1^n,G_3^n)_{\Omega}+3(G_2^n,G_3^n)_{\Omega} + 3(G_3^n,G_3^n)_{\Omega}\notag\\
&=&- ||G_2^n||_{0,\Omega}^2 + \tau \left(\mathcal{K}_{RK}(G_2^n) + \mathcal{L}_{RK}(G_1^n) +
\mathcal{L}_{RK}(G_2^n)\right) +3(G_3^n,G_3^n)_{\Omega}.
\label{eq:lem3:10:1}
\end{eqnarray}
Based on \eqref{eq:lem:3.9_1} and \eqref{eq:lem:3.9_3},
\begin{eqnarray*}
&&\mathcal{K}_{RK}(G_2^n) + \mathcal{L}_{RK}(G_1^n) \\
\leq && a_h(G_2^n, E_h^{n,1}, B_h^{n,1}; G_1^n) +a_h(G_1^n, E_h^{n,1}, B_h^{n,1}; G_2^n)\\
&& +C\left((1+\frac{\tau}{h})\sum_{\sharp=0}^{2} (||\xi_E^{n,\sharp}||_{0,\Omega_x} + ||\xi_B^{n,\sharp}||_{0,\Omega_x} + ||\xi_f^{n,\sharp}||_{0,\Omega}+h^{k+1}) +\tau^3 \right) (||G_1^n||_{0,\Omega} +||G_2^n||_{0,\Omega}).
\end{eqnarray*}
We now apply Lemma \ref{lemma:3.8} to estimate the first two terms on the right, and apply Cauchy-Schwartz inequality and \eqref{eq:G=f} to estimate the last term, and get
\begin{eqnarray}
&&\mathcal{K}_{RK}(G_2^n) + \mathcal{L}_{RK}(G_1^n)\notag\\
\leq &&
 C \left( 1 +\frac{\tau}{h} \right) \left( \sum_{\sharp=0}^{2} \ene^{n,\sharp} + h^{2k+2} \right) + C \tau^6 +
  \frac{C}{h} ||G_2^n||_{0,\Omega}^2 +  \frac{1}{16} (\stab_f^n + \stab_f^{n,1}). \label{eq:lem3:10:2}
\end{eqnarray}

In order to estimate $\mathcal{L}_{RK}(G_2^n)$ in \eqref{eq:lem3:10:1}, we apply
\eqref{eq:lem:3.9_2} of Lemma~\ref{lemma:3.9}. In particular, take $g = G_2^n$ in \eqref{eq:lem:3.9_2} and use \eqref{eq:G=f}, we have
 \begin{eqnarray}
\mathcal{L}_{RK}(G_2^n)
  &\leq&
  C\left((1+\frac{\tau}{h})\sum_{\sharp=0}^{2} (||\xi_E^{n,\sharp}||_{0,\Omega_x} + ||\xi_B^{n,\sharp}||_{0,\Omega_x} + ||\xi_f^{n,\sharp}||_{0,\Omega} + h^{k+1})+\tau^3 + \frac{||G_2^n||_{0,\Omega}}{h} \right) ||G_2^n||_{0,\Omega}  \notag\\
  &\leq&
   C(1+\frac{\tau}{h})\left(\sum_{\sharp=0}^{2}\ene^{n,\sharp} + h^{2k+2}\right) +C\tau^6
   +\frac{C}{h}||G_2^n||_{0,\Omega}^2. \label{eq:lem3:10:3}
  \end{eqnarray}

For $3(G_3^n,G_3^n)_{\Omega}=\tau \mathcal{L}_{RK}(G_3^n)$, we take
$g = G_3^n$ in \eqref{eq:lem:3.9_2} and obtain
 \begin{align*}
  3||G_3^n||_{0,\Omega}^2
  &\leq
C\tau  \left((1+\frac{\tau}{h})\sum_{\sharp=0}^{2} (||\xi_E^{n,\sharp}||_{0,\Omega_x} + ||\xi_B^{n,\sharp}||_{0,\Omega_x} + ||\xi_f^{n,\sharp}||_{0,\Omega}+ h^{k+1})+\tau^3 + \frac{||G_2^n||_{0,\Omega}}{h} \right) ||G_3^n||_{0,\Omega}\\
  &\leq
C \tau^2 (1+\frac{\tau}{h})^2 \left( \sum_{\sharp = 0}^2 \ene^{n,\sharp} + h^{2k+2} \right)
     +C \tau^8+ C \frac{\tau^2}{h^2} ||G_2^n||_{0,\Omega}^2 +  ||G_3^n||_{0,\Omega}^2.
 \end{align*}
Therefore, with a different constant $C$ we have,
  \begin{equation}
 \label{eq:lem3:10:4}
     3(G_3^n, G_3^n)_{\Omega}
     \leq
     C \tau^2 (1+\frac{\tau}{h})^2 \left( \sum_{\sharp = 0}^2 \ene^{n,\sharp} + h^{2k+2} \right)
     +C \tau^8+ C \frac{\tau^2}{h^2} ||G_2^n||_{0,\Omega}^2.
  \end{equation}
  Finally, we complete the proof by combing \eqref{eq:lem3:10:1}-\eqref{eq:lem3:10:4}.
\end{proof}

Now we are ready to establish the main error estimate result for the Vlasov solver.
\begin{theorem}
\label{thm:vlas}
Let $(f, E, B)$ be a sufficiently smooth exact solution to equations \eqref{eq1:vm1}.
 Let $(f_h^n, E_h^n, B_h^n)\in\mG_h^k\times \mU_h^k\times \mU_h^k$ be the solution to the scheme \eqref{eq:1:2:3} at time $t^n$
  with $k \geq \left \lceil \frac{d_x}{2} \right \rceil$.
Under the $L^\infty$-Assumption,
 there exists a positive constant $\gamma_1$,
  such that  for any $\frac{\tau}{h} \leq \gamma_1$,
 the following estimate holds for $\forall n: n+1\leq T/\tau$
\begin{equation}
  3||\xi_f^{n+1}||_{0,\Omega}^2 - 3||\xi_f^{n}||_{0,\Omega}^2
  \leq
  C \tau h^{2k+1} + C\tau^7 +
  C\tau
   \sum_{\sharp=0}^2 \ene^{n,\sharp}.
\end{equation}
\end{theorem}
\begin{proof}
Since the constant $C$ in the result $\eqref{eq:X4}$ is independent of $n, h, \tau$,
there exists a positive constant $\gamma_1$ independent of $n, h, \tau$,
such that $-1 + C \frac{\tau}{h} + C \frac{\tau^2}{h^2} \leq -\frac{1}{2}$ as long as $\frac{\tau}{h}\leq \gamma_1$.
Under such condition, we combine \eqref{eq:101} and the estimates in Proposition \ref{3.4} and Proposition \ref{lemma:3.10} and get
\begin{eqnarray}
\label{eq:X6}
  & & 3||\xi_f^{n+1}||_{0,\Omega}^2 - 3||\xi_f^{n}||_{0,\Omega}^2
  = \Xi_{1} + \Xi_{2} +\Xi_3 \nonumber\\
  &\leq&
  C \tau h^{2k+1} + C\tau^7 +
  C\tau
   \sum_{\sharp=0}^2 \ene^{n,\sharp}
   -
   \frac{1}{2}||G_2^n||_{0,\Omega}^2
   -
   \frac{\tau}{8}\left(3\stab_f^n + 3\stab_f^{n,1} + 14 \stab_f^{n,2}\right)
  \nonumber\\
  &\leq&
  C \tau h^{2k+1} + C\tau^7 +
  C\tau
   \sum_{\sharp=0}^2  \ene^{n,\sharp}.\notag
\end{eqnarray}
\end{proof}

\subsection{The Maxwell equations part}
\label{sec:3.3}

In this section, we estimate how error accumulates in the Maxwell solver. The procedure is parallel to the Vlasov part in section \ref{sec:3.2}.
We start with a decomposition of the error change in one time step
$$\xi_E^{n+1} - \xi_E^{n} = X_1^n + X_2^n + X_3^n,\quad \xi_B^{n+1} - \xi_B^{n} = Z_1^n + Z_2^n + Z_3^n,$$
where $X_1^n = \xi_E^{n,1}-\xi_E^{n}$, $X_2^n = 2\xi_E^{n,2}-\xi_E^{n,1}-\xi_E^{n}$, $X_3^n = \xi_E^{n+1} - 2\xi_E^{n,2} + \xi_E^n$, $Z_1^n = \xi_B^{n,1}-\xi_B^{n}$, $Z_2^n = 2\xi_B^{n,2}-\xi_B^{n,1}-\xi_B^{n}$, and $Z_3^n = \xi_B^{n+1} - 2\xi_B^{n,2} + \xi_B^n$. It is obvious that
\begin{equation}
||X_i^n||_{0,\Omega}^2 \leq C\sum_{\sharp=0}^2 ||\xi_E^{n,\sharp}||_{0,\Omega_x}^2, ||Z_i^n||_{0,\Omega}^2 \leq C\sum_{\sharp=0}^2 ||\xi_B^{n,\sharp}||_{0,\Omega_x}^2,\quad i=1, 2, 3.\label{eq:XZ=f}
\end{equation}
Based on
equations \eqref{eq:err2}, the following hold for any $U, V \in \mathcal{U}_h^k$.
\begin{eqnarray*}
  (X_2^n, U)_{\Omega_x} + (Z_2^n, V)_{\Omega_x} &=& \frac{\tau}{2}(\mathcal{R}(U,V)-\mathcal{Q}(U,V))\equiv \frac{\tau}{2}\mathcal{R}_{RK}(U,V), \\
  (X_3^n, U)_{\Omega_x} + (Z_3^n, V)_{\Omega_x} &=& \frac{\tau}{3}(2\mathcal{S}(U,V)-\mathcal{R}(U,V)-\mathcal{Q}(U,V))\equiv \frac{\tau}{3}\mathcal{S}_{RK}(U,V).
\end{eqnarray*}
In addition, one can verify  based on \eqref{eq:errEB} that
\begin{equation}
 \label{eq:M1}
3(||\xi_E^{n+1}||_{0,\Omega_x}^2 + ||\xi_B^{n+1}||_{0,\Omega_x}^2)
 - 3(||\xi_E^{n}||_{0,\Omega_x}^2 + ||\xi_B^{n}||_{0,\Omega_x}^2)
 =\Theta_1+\Theta_2+\Theta_3,
 \end{equation}
 where
 \begin{eqnarray}
 \Theta_i=&&\tau\left(\mathcal{Q}_i(\xi_E^n, \xi_B^n) + \mathcal{R}_i(\xi_E^{n,1}, \xi_B^{n,1}) + 4\mathcal{S}_i(\xi_E^{n,2}, \xi_B^{n,2})\right),\quad i=1, 2 \label{eq:theta1}\\
 \Theta_3=&&(X_2^n, X_2^n)_{\Omega_x} + 3(X_1^n, X_3^n)_{\Omega_x} + 3(X_2^n, X_3^n)_{\Omega_x} + 3(X_3^n, X_3^n)_{\Omega_x}  \notag\\
 & &+ (Z_2^n, Z_2^n)_{\Omega_x} + 3(Z_1^n, Z_3^n)_{\Omega_x} + 3(Z_2^n, Z_3^n)_{\Omega_x} + 3(Z_3^n, Z_3^n)_{\Omega_x}.\label{eq:theta2}
\end{eqnarray}
In particular, $\Theta_2$ depends on the spatial discretization, while $\Theta_1$ and $\Theta_3$ characterizes the contribution from the time discretization.  Similarly as in the Vlasov part, there are two stability mechanisms, with one being $\stab_{EB}^\star$ ($\star$ can be $n$ or $n, \sharp$) from the spatial discretization, and the other is related to $||X_2^n||^2$, $||Z_2^n||^2$ arising from  the third order Runge-Kutta time discretization.

\bigskip
Next we will estimate $\Theta_1$ and $\Theta_2$. Some estimates for  the spatial discretizations of the Maxwell part are summarized in Lemma \ref{lemma:new4.1}.
 \begin{lemma}
   \label{lemma:new4.1}
  For $\sharp = 0,1,2$, we have
  \begin{subequations}
     \begin{align}
 (i)& \;\;\   b_h(\xi_E^{n,\sharp}, \xi_B^{n,\sharp}, \xi_f^{n,\sharp}; \xi_E^{n,\sharp}, \xi_B^{n,\sharp})
     \leq C(||\xi_f^{n,\sharp}||_{0,\Omega}^2 + ||\xi_E^{n,\sharp}||_{0,\Omega_x}^2) -
     \frac{1}{2} 
     \stab_{EB}^{n,\sharp},\\
 (ii)& \;\;\ |b_h(\eta_E^{n,\sharp}, \eta_B^{n,\sharp}, \eta_f^{n,\sharp}; \xi_E^{n,\sharp}, \xi_B^{n,\sharp})|
      \leq
     Ch^{2k+1} + C||\xi_E^{n,\sharp}||^2_{0,\mathcal{E}_x}+
     \frac{1}{16} 
     \stab_{EB}^{n,\sharp}.
     \end{align}
   \end{subequations}
 \end{lemma}

\begin{proposition}
\label{prop:Theta12}
The following estimates hold for $\Theta_1$ and $\Theta_2$.
\begin{subequations}
 \begin{eqnarray}
 && \Theta_1 \leq C\tau( h^{2k+2}+\tau^6) + C\tau  \sum_{\sharp=0}^{2} \ene^{n,\sharp}\\
 && \Theta_2 \leq C\tau h^{2k+1}+ C\tau  \sum_{\sharp=0}^{2} \ene^{n,\sharp}
  -\frac{7}{16}\tau (\stab_{EB}^{n}+\stab_{EB}^{n,1}+4\stab_{EB}^{n,2}).
 \end{eqnarray}
 \end{subequations}
\end{proposition}
\begin{proof}
 Recall that $\Theta_1=\tau\left(\mathcal{Q}_1(\xi_E^n, \xi_B^n) + \mathcal{R}_1(\xi_E^{n,1}, \xi_B^{n,1}) + 4\mathcal{S}_1(\xi_E^{n,2}, \xi_B^{n,2})\right)$, with similarity, we only estimate $\mathcal{S}_1(\xi_E^{n,2}, \xi_B^{n,2})$. Applying Cauchy-Schwarz inequality, Lemma \ref{lemma:1}, Lemma \ref{lemma:2} - (1), one gets
 \begin{eqnarray*}
 \mathcal{S}_1(\xi_E^{n,2}, \xi_B^{n,2})
& = &\left(\frac{3\eta_E^{n+1}-\eta_E^n-2\eta_E^{n,2}+3T_E^n(x)}{2\tau} , \xi_E^{n,2}\right)_{\Omega_x}
  + \left(\frac{3\eta_B^{n+1}-\eta_B^n-2\eta_B^{n,2}+3T_B^n(x)}{2\tau}, \xi_B^{n,2}\right)_{\Omega_x}\\
&\leq &
 C( ||\xi_E^{n,2}||_{0,\Omega_x}^2 + ||\xi_B^{n,2}||_{0,\Omega_x}^2 + h^{2k+2} +\tau^6)
 \leq C(\ene^{n,2} + h^{2k+2} +\tau^6).
  \end{eqnarray*}

To bound $\Theta_2=\tau\left(\mathcal{Q}_2(\xi_E^n, \xi_B^n) + \mathcal{R}_2(\xi_E^{n,1}, \xi_B^{n,1}) + 4\mathcal{S}_2(\xi_E^{n,2}, \xi_B^{n,2})\right)$, with similarity, we only need to look at  $\mathcal{Q}_2(\xi_E^n, \xi_B^n)$. By definition, and Lemma \ref{lemma:new4.1},
\begin{align*}
 \mathcal{Q}_2(\xi_E^n, \xi_B^n)
&= b_h(e_E^n, e_B^n, e_f^n; \xi_E^n, \xi_B^n)
 = b_h(\xi_E^n, \xi_B^n, \xi_f^n; \xi_E^n, \xi_B^n)
  -b_h(\eta_E^n, \eta_B^n, \eta_f^n; \xi_E^n, \xi_B^n)\\
 &\leq
 Ch^{2k+1} + C(||\xi_f^n||_{0,\Omega}^2 + ||\xi_E^n||_{0,\Omega_x}^2) -
 \frac{7}{16} \stab_{EB}^n\\
  &\leq
 Ch^{2k+1} + C \ene^n - \frac{7}{16} \stab_{EB}^n.
\end{align*}
\end{proof}

\bigskip
Next we will estimate $\Theta_3$. One key is to use $-||X_2^n||_{0,\Omega_x}^2-||Z_2^n||_{0,\Omega_x}^2$ to control $C\frac{\tau}{h}(||X_2^n||_{0,\Omega_x}^2+||Z_2^n||_{0,\Omega_x}^2)$. To make the analysis easy to follow, we first present
some preparatory results in two lemmas.

\begin{lemma}
 \label{lemma:4.4}
For any $U, V \in \mU_h^k$, we have
\begin{subequations}
  \begin{align}
    \mathcal{S}_{RK}(U, V)
   &\leq
      C(h^{k+1} + \tau h^k+ \tau^3)(||U||_{0,\Omega_x} + ||V||_{0,\Omega_x})
   + b_h(X_2^n, Z_2^n, G_2^n; U, V)\label{eq:lem4:4_1}\\
   &\leq
   C(h^{k+1} + \tau h^k+ \tau^3+ \frac{1}{h} (||Z_2^n||_{0,\Omega_x} + ||X_2^n||_{0,\Omega_x}))(||U||_{0,\Omega_x} + ||V||_{0,\Omega_x})\notag   \\
   &  + C ||G_2^n||_{0,\Omega} ||U||_{0,\Omega_x}.\label{eq:lem4:4_2}\\
\mathcal{R}_{RK}(U, V)
&\leq
  C(h^{k+1} + \tau h^k)(||U||_{0,\Omega_x} + ||V||_{0,\Omega_x})
   + b_h(X_1^n, Z_1^n, G_1^n; U, V).\label{eq:lem4:4_3}
 \end{align}
 \end{subequations}
 \end{lemma}

\begin{lemma}
\label{lemma:new4.2}
\begin{eqnarray*}
  && b_h( X_1^{n}, Z_1^{n}, G_1^{n}; X_2^n, Z_2^n) +
   b_h( X_2^{n}, Z_2^{n}, G_2^{n}; X_1^n, Z_1^n)\nonumber\\
   &&  \leq
  \frac{C}{h}(||X_2^n||_{0,\Omega_x}^2 + ||Z_2^n||_{0,\Omega_x}^2)        
  + C\sum_{j=1}^2 (||G_j^n||_{0,\Omega}^2 + ||X_j^n||_{0,\Omega_x}^2)+ \frac{1}{16} ( \stab_{EB}^{n} + \stab_{EB}^{n,1} ).
\end{eqnarray*}
\end{lemma}

Now we get ready to estimate $\Theta_3$.
\begin{proposition}
\label{prop:Theta3}
\begin{eqnarray}
\Theta_3
  &\leq&
   C\tau( h^{2k+2} + \tau^2 h^{2k} + \tau^6)
  +\left( -1 + C\frac{\tau}{h} + C\frac{\tau^2}{h^2}\right) (||X_2^n||_{0,\Omega_x}^2 + ||Z_2^n||_{0,\Omega_x}^2)\notag\\
  &&+ \frac{\tau}{16}\left( \stab_{EB}^{n} + \stab_{EB}^{n,1} \right)+C \tau  \sum_{\sharp=0}^{2} \ene^{n,\sharp}.\label{eq:theta3:0}
\end{eqnarray}
\end{proposition}

\begin{proof}
First note that
\begin{eqnarray}
 \Theta_3&=&-||X_2^n||_{0,\Omega_x}^2-||Z_2^n||_{0,\Omega_x}^2 \notag\\
 &&+ \tau (\mathcal{R}_{RK}(X_2^n, Z_2^n) +\mathcal{S}_{RK}(X_1^n, Z_1^n)+\mathcal{S}_{RK}(X_2^n, Z_2^n))
 + 3||X_3^n||_{0,\Omega_x}^2+3||Z_3^n||_{0,\Omega_x}^2.\label{eq:theta3:1}
\end{eqnarray}

Based on \eqref{eq:lem4:4_1}, \eqref{eq:lem4:4_3}, and Lemma \ref{lemma:new4.2}, in addition to \eqref{eq:G=f} and \eqref{eq:XZ=f} we have
\begin{align}
&\mathcal{R}_{RK}(X_2^n, Z_2^n) +\mathcal{S}_{RK}(X_1^n, Z_1^n)\notag\\
&\leq
b_h(X_2^n, Z_2^n, G_2^n; X_1^n, Z_1^n)   + b_h(X_1^n, Z_1^n, G_1^n; X_2^n, Z_2^n)
+ C(h^{k+1} + \tau h^k+ \tau^3)(||X_1^n||_{0,\Omega_x} + ||Z_1^n||_{0,\Omega_x})\notag\\
&+  C(h^{k+1} + \tau h^k)(||X_2^n||_{0,\Omega_x} + ||Z_2^n||_{0,\Omega_x})\notag\\
&\leq \frac{C}{h}(||X_2^n||_{0,\Omega_x}^2 + ||Z_2^n||_{0,\Omega_x}^2)
+ \frac{1}{16} ( \stab_{EB}^{n} + \stab_{EB}^{n,1} )\notag\\
 &  + C(h^{2k+2}+\tau^2 h^{2k} + \tau^6)
 + C\sum_{j=1}^2 (||G_j^n||_{0,\Omega}^2 + ||X_j^n||_{0,\Omega_x}^2 +||Z_j^n||_{0,\Omega_x}^2 )\notag\\
 &\leq \frac{C}{h}(||X_2^n||_{0,\Omega_x}^2 + ||Z_2^n||_{0,\Omega_x}^2)
+ \frac{1}{16} ( \stab_{EB}^{n} + \stab_{EB}^{n,1} ) + C(h^{2k+2}+\tau^2 h^{2k} + \tau^6)
 + C\sum_{j=1}^2 \ene^{n,\sharp}.
\end{align}
Next we estimate $\mathcal{S}_{RK}(X_2^n, Z_2^n)$ by using \eqref{eq:lem4:4_2} in Lemma \ref{lemma:4.4}.
  \begin{align}
    \mathcal{S}_{RK}(X_2^n, Z_2^n)
   &\leq
   C(h^{k+1} + \tau h^k+ \tau^3+ \frac{1}{h} (||Z_2^n||_{0,\Omega_x} + ||X_2^n||_{0,\Omega_x}))(||X_2^n||_{0,\Omega_x} + ||Z_2^n||_{0,\Omega_x})   + C ||G_2^n||_{0,\Omega} ||X_2^n||_{0,\Omega_x}\notag\\
   &\leq
   C(h^{2k+2} + \tau^2 h^{2k}+ \tau^6) + C\sum_{j=1}^2 \ene^{n,\sharp} + \frac{C}{h} (||Z_2^n||^2_{0,\Omega_x} + ||X_2^n||^2_{0,\Omega_x}).
 \end{align}
Finally we turn to
$3||X_3^n||_{0,\Omega_x}^2+3||Z_3^n||_{0,\Omega_x}^2$ in \eqref{eq:theta3:1}. By applying \eqref{eq:lem4:4_2} in Lemma \ref{lemma:4.4},
\begin{eqnarray*}
&&3||X_3^n||_{0,\Omega_x}^2+3||Z_3^n||_{0,\Omega_x}^2
=\tau\mathcal{S}_{RK}(X_3^n, Z_3^n)\\
&\leq&
 C\tau(h^{k+1} + \tau h^k+ \tau^3+ \frac{1}{h} (||Z_2^n||_{0,\Omega_x} + ||X_2^n||_{0,\Omega_x}))(||X_3^n||_{0,\Omega_x} + ||Z_3^n||_{0,\Omega_x})   + C \tau ||G_2^n||_{0,\Omega} ||X_3^n||_{0,\Omega_x}\\
 &\leq&
 C(\tau^2 h^{2k+2} + \tau^4 h^{2k}+ \tau^8)+ C\frac{\tau^2}{h^2} (||Z_2^n||^2_{0,\Omega_x} + ||X_2^n||^2_{0,\Omega_x})+C\tau^2||G_2^n||^2_{0,\Omega}
+||X_3^n||^2_{0,\Omega_x} + ||Z_3^n||^2_{0,\Omega_x},
\end{eqnarray*}
therefore, with a different value of $C$, we have
\begin{equation}
3||X_3^n||_{0,\Omega_x}^2+3||Z_3^n||_{0,\Omega_x}^2
\leq
 C(\tau^2 h^{2k+2} + \tau^4 h^{2k}+ \tau^8)+ C\frac{\tau^2}{h^2} (||Z_2^n||^2_{0,\Omega_x} + ||X_2^n||^2_{0,\Omega_x})+C\tau^2||G_2^n||^2_{0,\Omega}.
 \label{eq:theta3:2}
\end{equation}
Now by combining the results in \eqref{eq:theta3:1} - \eqref{eq:theta3:2} and \eqref{eq:G=f}, we can conclude the estimate for $\Theta_3$ in \eqref{eq:theta3:0}.
\end{proof}

The main error estimate result for the Maxwell solver is now established as following.
\begin{theorem}
 \label{thm:max}
  Let $(f, E, B)$ be a sufficiently smooth exact solution to equations \eqref{eq1:vm1}.
  Let $(f_h^n, E_h^n, B_h^n)\in\mG_h^k\times \mU_h^k\times \mU_h^k$ be the solution to the scheme \eqref{eq:1:2:3} at time $t^n$.
There exists a positive constant $\gamma_2$,
 such that for $\frac{\tau}{h} \leq \gamma_2$,
 the following estimate holds for $\forall n: n+1\leq  T/\tau$
 \begin{align}
 \label{eq:M21}
  & 3(||\xi_E^{n+1}||_{0,\Omega_x}^2 + ||\xi_B^{n+1}||_{0,\Omega_x}^2) -
    3(||\xi_E^{n}||_{0,\Omega_x}^2 + ||\xi_B^{n}||_{0,\Omega_x}^2)
  \leq
  C\tau h^{2k+1} + C\tau^7 + C \tau
  \sum_{\sharp=0}^{2} \ene^{n,\sharp}.
 \end{align}
\end{theorem}
\begin{proof}
Since the constant $C$ in the estimate \eqref{eq:theta3:0} is independent of $n, h, \tau$, there exists a positive constant $\gamma_2$ independent of
$n, h, \tau$, such that such that $-1 + C\frac{\tau}{h} + C\frac{\tau^2}{h^2} \leq -\frac{1}{2}$
 as long as  $\frac{\tau}{h} \leq \gamma_2$. Under this condition on the time step $\tau$, we also have $\tau^2 h^{2k}\leq \gamma_2^2 h^{2k+2}$.
Now we combine the estimates for $\Theta_i$, $i=1, 2, 3$ in Proposition \ref{prop:Theta12} and Proposition \ref{prop:Theta3}, and get
\begin{align*}
&3(||\xi_E^{n+1}||_{0,\Omega_x}^2 + ||\xi_B^{n+1}||_{0,\Omega_x}^2) -
    3(||\xi_E^{n}||_{0,\Omega_x}^2 + ||\xi_B^{n}||_{0,\Omega_x}^2)
    =\Theta_1+\Theta_2+\Theta_3\\
&\leq
   C\tau h^{2k+1} + C\tau^7+ C\tau \sum_{\sharp=0}^{2} \ene^{n,\sharp} -\frac{1}{2} (||X_2^n||_{0,\Omega_x}^2 + ||Z_2^n||_{0,\Omega_x}^2)
  - \frac{\tau}{8}\left(3\stab_{EB}^{n} + 3\stab_{EB}^{n,1} +14 \stab_{EB}^{n,2}\right)\\
  &\leq C\tau h^{2k+1} + C\tau^7+ C\tau \sum_{\sharp=0}^{2} \ene^{n,\sharp}.
\end{align*}
\end{proof}
\begin{remark}
 Unlike in Theorem \ref{thm:vlas}, the a priori assumption about the $L^{\infty}$ norm of the error in the
 magnetic and electric fields, together  with the condition of $k \geq \left \lceil \frac{d_x}{2} \right \rceil$, are not needed in 
  Theorem \ref{thm:max}. This difference is due to the nonlinear coupling terms in the Vlasov equation.
\end{remark}

\subsection{Proof of the main result: Theorem \ref{theorem:1}}
\label{sec:3.4}

The following lemma is the final preparation.
\begin{lemma}
 Suppose $\frac{\tau}{h} \leq \alpha$, where $\alpha$ is a positive constant independent of $n, h, \tau$. Then under the $L^\infty$-Assumption, the following inequalities are satisfied
\label{lemma:5.1}
 \begin{eqnarray*}
  ||\xi_f^{n,1}||_{0,\Omega}^2 &\leq& Ch^{2k+2} + C(||\xi_f^{n}||_{0,\Omega}^2 + \tau^2 ||\xi_E^{n}||_{0,\Omega_x}^2 +
               \tau^2 ||\xi_B^{n}||_{0,\Omega_x}^2)  \\
  ||\xi_f^{n,2}||_{0,\Omega}^2 &\leq& Ch^{2k+2} + C(||\xi_f^{n}||_{0,\Omega}^2 + ||\xi_f^{n,1}||_{0,\Omega}^2 + \tau^2 ||\xi_E^{n,1}||_{0,\Omega_x}^2 +
               \tau^2 ||\xi_B^{n,1}||_{0,\Omega_x}^2)  \\
  ||\xi_E^{n,1}||_{0,\Omega_x}^2 &\leq& Ch^{2k+2} + C(\tau^2 ||\xi_f^{n}||_{0,\Omega}^2 + ||\xi_E^{n}||_{0,\Omega_x}^2 +
               ||\xi_B^{n}||_{0,\Omega_x}^2)  \\
  ||\xi_E^{n,2}||_{0,\Omega_x}^2 &\leq& Ch^{2k+2} + C(\tau^2 ||\xi_f^{n,1}||_{0,\Omega_x}^2 + ||\xi_E^{n}||_{0,\Omega_x}^2+ ||\xi_E^{n,1}||_{0,\Omega_x}^2 +
               ||\xi_B^{n,1}||_{0,\Omega_x}^2) \\
  ||\xi_B^{n,1}||_{0,\Omega_x}^2 &\leq& Ch^{2k+2} + C( ||\xi_E^{n}||_{0,\Omega_x}^2 +
               ||\xi_B^{n}||_{0,\Omega_x}^2)  \\
  ||\xi_B^{n,2}||_{0,\Omega_x}^2 &\leq& Ch^{2k+2} + C(||\xi_E^{n,1}||_{0,\Omega_x}^2 + ||\xi_B^{n}||_{0,\Omega_x}^2+
               ||\xi_B^{n,1}||_{0,\Omega_x}^2).
 \end{eqnarray*}
 A direct consequence of these inequalities is
  \begin{equation}
 \label{eq:VM2}
   \sum_{\sharp=0}^{2}
    \ene^{n, \sharp}
  \leq
  C \ene^{n} + Ch^{2k+2}.
 \end{equation}
\end{lemma}

Now we are ready for the proof of Theorem~\ref{theorem:1}.
\begin{proof}
 First we make the a priori assumption for the $L^\infty$ error of the magnetic and electric fields as in $L^\infty$-Assumption.
 Based on Theorem \ref{thm:vlas} and Theorem \ref{thm:max}, for any $\frac{\tau}{h}\leq \gamma:= \min(\gamma_1, \gamma_2)$,  we get
 \begin{eqnarray}
   \label{eq:VM1}
   3 \ene^{n+1} - 3 \ene^{n}
  &\leq&
  C\tau^7 + C\tau h^{2k+1}  + C \tau
  \sum_{\sharp=0}^{2} \ene^{n, \sharp}
  \leq
  \hat{C}\left(\tau^7 + \tau h^{2k+1}  + \tau \ene^{n}
   \right).
 \end{eqnarray}
 Here $\eqref{eq:VM2}$ is used to get the last inequality.
 Let $\Upsilon_{n} =  \ene^{n}
  /(1 + \frac{\hat{C}\tau}{3})^{n} $,
  then $\eqref{eq:VM1}$ leads to
  $\Upsilon_{n} - \Upsilon_{n-1} \leq
   \frac{\hat{C}(\tau^7 + \tau h^{2k+1})}{3(1 + \frac{\hat{C}}{3}\tau)^{n}}
  $.
We now sum up $\Upsilon_{\sharp} - \Upsilon_{\sharp - 1}$ and use $\Upsilon_{0} = 0$ to obtain
  \begin{equation}
   \Upsilon_{n}
   \leq \sum_{\sharp = 1}^n \frac{\hat{C}(\tau^7 + \tau h^{2k+1})}{3(1+ \frac{\hat{C}}{3}\tau)^{\sharp}}
   \leq
   \tau^6 + h^{2k+1}.
  \end{equation}
  Therefore
  \begin{equation}
   \ene^{n}
   =
   (1 + \frac{\hat{C}\tau}{3})^n \Upsilon_{n}
   \leq
    e^{\frac{\hat{C}n\tau}{3}} (\tau^6 + h^{2k+1})
   \leq
    e^{\frac{\hat{C}T}{3}} (\tau^6 + h^{2k+1}),
  \end{equation}
   for any $n\leq T/\tau$. This can also be written as
  \begin{equation}
  \label{eq:VM8}
   ||\xi_f^{n}||_{0,\Omega}^2 +
   ||\xi_E^{n}||_{0,\Omega_x}^2 + ||\xi_B^{n}||_{0,\Omega_x}^2
   \leq
   C (\tau^6 + h^{2k+1}).
  \end{equation}
  All that remains is to prove that the $L^\infty$-Assumption is actually satisfied. This will be established by mathematical induction.

  For $n=0$, $\xi_f^{0} = \xi_E^{0} = \xi_B^{0} = 0$ and the approximation property \eqref{eq:approx} implies that $||e_E^{0}||_{0,\infty,\Omega_x}$, $||e_B^{0}||_{0,\infty,\Omega_x} \leq Ch$,
  where $C$ is independent of $n, h, \tau$.
  Furthermore, Lemma~\ref{lemma:5.1} shows
  $||\xi_f^{0,\sharp}||_{0,\Omega}$, $||\xi_E^{0,\sharp}||_{0,\Omega_x}$,
  $||\xi_B^{0,\sharp}||_{0,\Omega_x} \leq Ch^{k+1}$ for $\sharp = 1,2$.
  Thus
  \begin{equation*}
    ||e_E^{0,\sharp}||_{0,\infty,\Omega_x}
    \leq ||\xi_E^{0,\sharp}||_{0,\infty,\Omega_x}+||\eta_E^{0,\sharp}||_{0,\infty,\Omega_x}
    \leq
    C h^{-\frac{d_x}{2}} ||\xi_E^{0,\sharp}||_{0,\Omega_x} +
    C h^{k+1}
    \leq
    C h^{k+1-\frac{d_x}{2}}
    \leq C h
    .
  \end{equation*}
  Here the inverse inequality $\eqref{eq:inv}$ and the condition $k\geq \left \lceil \frac{d_x}{2} \right \rceil$ are used.
  Similarly
  $||e_B^{0,\sharp}||_{0,\infty,\Omega_x}\leq C h$.
  Suppose $||e_B^{n}||_{0,\infty,\Omega_x}, ||e_E^{n}||_{0,\infty,\Omega_x}\leq C h$ for all $n\leq m$, therefore
  inequality \eqref{eq:VM8} with $n=m+1$ is satisfied.
  Combining this with Lemma~\ref{lemma:5.1} implies  $||\xi_f^{m+1,\sharp}||_{0,\Omega}, ||\xi_E^{m+1,\sharp}||_{0,\Omega_x},
  ||\xi_B^{m+1,\sharp}||_{0,\Omega_x} \leq C (\tau^3 + h^{k+\frac{1}{2}})$ for $\sharp =0,1,2$. Thus
  \begin{eqnarray}
  \label{eq:VM9}
    ||e_E^{m+1,\sharp}||_{0,\infty,\Omega_x}
    &\leq&
    C h^{-\frac{d_x}{2}} ||\xi_E^{m+1,\sharp}||_{0,\Omega_x} +
    C h^{k+1}
    \leq
    C h^{-\frac{d_x}{2}} (\tau^3 + h^{k+\frac{1}{2}}) + Ch^{k+1}
    \nonumber\\
    &\leq&
    C h^{-\frac{d_x}{2}} \tau^3  + Ch^{k+\frac{1}{2} -\frac{d_x}{2}}
    \leq
    C \gamma^3 h^{3 -\frac{d_x}{2}} + Ch^{k+\frac{1}{2} -\frac{d_x}{2}}
    \leq
    C h,
  \end{eqnarray}
  for $k \geq \left \lceil \frac{d_x + 1}{2} \right \rceil$ (recall $d_x\leq 3$),  under the condition $\frac{\tau}{h} \leq \gamma$.
 Up to now the a priori assumption is established, and this completes the proof of the main theorem of this paper.
\end{proof}
  \begin{remark}
    The condition $k \geq  \left \lceil \frac{d_x + 1}{2} \right \rceil$ is needed only for the last inequality of $\eqref{eq:VM9}$ throughout the proof, while
    for  all the other results, the requirement $k \geq \left \lceil \frac{d_x}{2} \right \rceil$ is enough.
  \end{remark}

\section{Proofs of Some Lemmas}
\label{sec:4}

In this section, we will provide the proofs of some lemmas in section \ref{sec:3}.

\subsection{Proof of Lemma \ref{lemma:2}}
\label{proof:lemma:2}

\underline{To get \eqref{eq:2.2.1}}, we start with $f^{n,1}$ and $f^{n,2}$ given in
\eqref{eq:exact_new},  and get
\begin{align}
 &d_0 f^n + d_1 f^{n,1}+d_2 f^{n,2}+d_3 f^{n+1}  \nonumber\\
 =
 &d_3 (f^{n+1} - f^{n}) + \tau \left(d_1 + \frac{d_2}{2} \right)\partial_t f^{n} + \frac{\tau^2}{4} d_2 \partial_t^2 f^n
   - \frac{\tau^3}{4} d_2 (\partial_t E^n + v\times \partial_t B^n)\cdot \nabla_v(\partial_t f^n)\notag\\
 =&\tau d_3\partial_t f(x, v, t^\star) + \tau \left(d_1 + \frac{d_2}{2} \right)\partial_t f^{n} + \frac{\tau^2}{4} d_2 \partial_t^2 f^n
   - \frac{\tau^3}{4} d_2 (\partial_t E^n + v\times \partial_t B^n)\cdot \nabla_v(\partial_t f^n),
   \end{align}
   for some $t^\star\in[t^n, t^{n+1}]$. Here we have used $d_0 + d_1 + d_2 + d_3=0$. Therefore, with $I$ as the identity operator,
\begin{equation*}
 d^n_f
 =(\Pi^k-I)\left(\tau d_3\partial_t f(x,v,t^\star) + \tau \left(d_1 + \frac{d_2}{2} \right)\partial_t f^{n} + \frac{\tau^2}{4} d_2 \partial_t^2 f^n
   - \frac{\tau^3}{4} d_2 (\partial_t E^n + v\times \partial_t B^n)\cdot \nabla_v(\partial_t f^n)\right).
\end{equation*}
 For sufficiently smooth solution,  there is
\begin{equation*}
 ||d_f^n||_{0,\Omega}
 \leq
 C\tau h^{k+1} \max_{\forall t\in [0, T]}
 \left(
 ||\partial_t f||_{k+1,\Omega} + \tau ||\partial_t^2 f||_{k+1,\Omega} + \tau^2 || (\partial_t E + v\times \partial_t B)\cdot \nabla_v(\partial_t f) ||_{k+1,\Omega}
 \right).
\end{equation*}
Recall that $\tau \leq 1$, we further have $||d_f^n||_{0,\Omega} \leq C\tau h^{k+1}$.
Similarly one can show $||d_f^n||_{0,\mathcal{E}} \leq C\tau h^{k+\frac{1}{2}}$.
These two estimates will lead to the upper bound in $\eqref{eq:2.2.1}$.
The proof of the results for $E$ and $B$ can be proceeded similarly.

\bigskip
\noindent\underline{To get \eqref{eq:2.2.1_1}}, based on definition
\begin{eqnarray*}
  & & |a_h(d^n_f, E_h^{n,s}, B_h^{n,s}; g)|               \nonumber\\
  &\leq& \int_{\Omega} |d^n_f v \cdot \nabla_x g|dx dv +
   \int_{T_h^v}\int_{\mathcal{E}_x} \left(  |\{d^n_f v \}_x + \frac{|v\cdot n_x|}{2} [d^n_f]_x| \right) |[g]_x| ds_x dv  
   + \int_\Omega |d^n_f (E_h^{n,s}+v\times B_h^{n,s}) \nabla_v g| dxdv\\
  && + \int_{T_h^x} \int_{\mathcal{E}_v}  |\left(\{d^n_f (E_h^{n,s}+v\times B_h^{n,s}) \}_v
  + \frac {|(E_h^{n,s} + v\times B_h^{n,s}) \cdot n_v|}{2} [d^n_f]_v\right)\cdot  [g]_v|   ds_v dx             \nonumber\\
  &\leq&
  C \left(||d^n_f||_{0,\Omega} ||\nabla_x g||_{0,\Omega} +  ||d^n_f||_{0,T_h^v \times \mathcal{E}_x} ||g||_{0,T_h^v \times \mathcal{E}_x} \right)       \nonumber\\
  & &+
  C \left( || E_h^{n,s}||_{0,\infty,\Omega_x} + || B_h^{n,s}||_{0,\infty,\Omega_x} \right) \left(||d^n_f||_{0,\Omega}
   ||\nabla_v g||_{0,\Omega} + ||d^n_f||_{0,T_h^x \times \mathcal{E}_v}
   ||g||_{0,T_h^x \times \mathcal{E}_v}  \right)               \nonumber\\
  &\leq& \frac{C}{h}  ||g||_{0,\Omega} \left( ||d^n_f||_{0,\Omega} + h^{\frac{1}{2}} ||d^n_f||_{0,T_h^v \times \mathcal{E}_x} \right)  \nonumber\\
  & &+
  \frac{C}{h}  ||g||_{0,\Omega} \left( || E_h^{n,s}||_{0,\infty,\Omega_x} + || B_h^{n,s}||_{0,\infty,\Omega_x} \right)
     \left( ||d^n_f||_{0,\Omega} + h^{\frac{1}{2}} ||d^n_f||_{0,T_h^x \times \mathcal{E}_v} \right) .
\end{eqnarray*}
By the a priori assumption, there is
 \begin{align}
 \label{eq:lem:6:proof}
  &||E_h^{n,s}||_{0,\infty,\Omega_x}
   \leq ||e_E^{n,s}||_{0,\infty,\Omega_x} + ||E^{n,s}||_{0,\infty,\Omega_x}
   \leq Ch + C \leq C,
 \end{align}
and  similarly,
 \begin{equation}
 \label{eq:lem:6:proof:1}
  ||B_h^{n,s}||_{0,\infty,\Omega_x} \leq C.
 \end{equation}
Finally we apply \eqref{eq:2.2.1} and conclude
\begin{align}
   |a_h(d^n_f, E_h^{n,s}, B_h^{n,s}; g)|  \leq C\frac{\tau}{h} ||g||_{0,\Omega} h^{k+1} 
   \leq C \frac{\tau}{h} (h^{2k+2} + ||g||_{0,\Omega}^2).
\end{align}


\bigskip
\noindent\underline{To get \eqref{eq:2.2.1_4}},
\begin{eqnarray*}
 & & b_h(d_E^n, d_B^n, d_f^n; U, V) \\
 &=& \int_{\Omega_x} d_B^n \cdot (\nabla \times U) dx - \int_{\Omega_x} d_E^n \cdot (\nabla \times V) dx - \int_{\Omega_x} (\int_{\Omega_v} d_f^n v dv)\cdot U dx \nonumber\\
 & & + \int_{\mathcal{E}_x} (\{ d_B^n \}_x - \frac{1}{2}[d_E^n]_{tan})\cdot [U]_{tan} ds_x - \int_{\mathcal{E}_x} (\{ d_E^n \}_x + \frac{1}{2}[d_B^n]_{tan})\cdot [V]_{tan} ds_x \nonumber\\
 &=& - \int_{\Omega_x} (\int_{\Omega_v} d_f^n v dv)\cdot U dx
     + \int_{\mathcal{E}_x} (\{ d_B^n \}_x - \frac{1}{2}[d_E^n]_{tan})\cdot [U]_{tan} ds_x - \int_{\mathcal{E}_x} (\{ d_E^n \}_x + \frac{1}{2}[d_B^n]_{tan})\cdot [V]_{tan} ds_x,\nonumber
\end{eqnarray*}
 where the second equality is due to $\nabla \times U, \nabla \times V \in \mathcal{U}_h^k$, which leads to
 $\int_{\Omega_x} d_B^n \cdot (\nabla \times U)dx = \int_{\Omega_x} d_E^n \cdot (\nabla \times V)dx = 0$.
 We now can apply inverse inequality \eqref{eq:inv} and \eqref{eq:2.2.1} to obtain \eqref{eq:2.2.1_4}.

\subsection{Proof of Lemma \ref{3.1}}
\label{proof:3.1}

\underline{To get \eqref{eq:3.1.2}}, based on definition, for any $g \in \mathcal{G}_h^k$, $U, V \in \mathcal{U}_h^k$

  \begin{eqnarray}
   \label{eq:3.1.3}
   a_{h,1}(g; g) &=& \int_{\Omega} g v\cdot \nabla_x g dxdv - \int_{T_h^v}\int_{\mathcal{E}_x}\left(\{g v\}_x + \frac{|v\cdot n_x|}{2}[g]_x \right) \cdot [g]_x ds_x dv               \nonumber\\
   &=& \frac{1}{2} \int_{T_h^v} \sum_{K_x \in T_h^x} \int_{\partial K_x} v\cdot n_x  g^2 ds_x dv - \int_{T_h^v}\int_{\mathcal{E}_x}\left(\{g v\}_x + \frac{|v\cdot n_x|}{2}[g]_x \right) \cdot [g]_x ds_x dv      \nonumber\\
   &=&  \int_{T_h^v}\int_{\mathcal{E}_x} \left( \frac{1}{2} v \cdot [g^2]_x - \left(\{g v\}_x + \frac{|v\cdot n_x|}{2}[g]_x \right)\cdot [g]_x \right) ds_x dv \nonumber\\
   &=&  \int_{T_h^v}\int_{\mathcal{E}_x} \left(\left( \frac{1}{2}  [g^2]_x - \{g \}_x
   [g]_x \right)\cdot v - \frac{|v\cdot n_x|}{2}|[g]_x|^2 \right) ds_x dv          \nonumber\\
   &=& -\frac{1}{2} \int_{T_h^v}\int_{\mathcal{E}_x}|v\cdot n_x| |[g]_x|^2 ds_x dv.
  \end{eqnarray}

Here the property of jump and average in $\eqref{eq:jump1}$ is used to get the last equality in \eqref{eq:3.1.3}.
Similarly, there is $ a_{h,2}(g, U, V; g)= -\frac{1}{2} \int_{T_h^x}\int_{\mathcal{E}_v} |(U + v\times V) \cdot n_v|  |[g]_v|^2  ds_v dx$.
Finally, taking $g = \xi_f^{n,\sharp}$, $U = E_h^{n,\sharp}$, $V = B_h^{n,\sharp}$ leads to \eqref{eq:3.1.2}.

\bigskip
\noindent\underline{To get \eqref{eq:3.1.2_1}}, we will proceed the proof in two steps.

\smallskip
\noindent {\sf Step 1: To estimate $a_{h,1}(\eta_f^{n,\sharp};\xi_f^{n,\sharp})$}.
By definition,
 \begin{equation}
  a_{h,1}(\eta_f^{n,\sharp};\xi_f^{n,\sharp}) = \int_{\Omega} \eta_f^{n,\sharp} v\cdot \nabla_x \xi_f^{n,\sharp}dx dv
  - \int_{T_h^v}\int_{\mathcal{E}_x} \left(\{\eta_f^{n,\sharp}v \}_x + \frac{|v\cdot n_x|}{2}[\eta_f^{n,\sharp}]_x   \right)\cdot [\xi_f^{n,\sharp}]_x ds_x dv.
 \end{equation}
 Let $v_0$ be the $L^2$ projection of the function $v$ onto the piecewise constant space with respect to $T_h^v$, then
 \begin{eqnarray}
  \int_{\Omega} \eta_f^{n,\sharp} v\cdot \nabla_x \xi_f^{n,\sharp}dx dv
   &=&
  \int_{\Omega} \eta_f^{n,\sharp} (v-v_0)\cdot \nabla_x \xi_f^{n,\sharp}dx dv + \int_{\Omega} \eta_f^{n,\sharp} v_0\cdot \nabla_x \xi_f^{n,\sharp}dx dv      \nonumber\\
   &=& \int_{\Omega} \eta_f^{n,\sharp} (v-v_0)\cdot \nabla_x \xi_f^{n,\sharp}dx dv.
 \end{eqnarray}
 The last equality is satisfied because $v_0\cdot \nabla_x \xi_f^{n,\sharp} \in \mathcal{G}_h^k$ and the $L^2$ projection of $\eta_f^{n,\sharp}$ onto $\mathcal{G}_h^k$ vanishes.
We further have
 \begin{align}
 \label{3.5}
  \left|\int_{\Omega} \eta_f^{n,\sharp} v\cdot \nabla_x \xi_f^{n,\sharp}dx dv\right| &\leq ||v-v_0||_{0,\infty,\Omega_v} \sum_{K\in T_h} (h_{K_x}^{-1}||\eta_f^{n,\sharp}||_{0,K}) (h_{K_x}||\nabla_x \xi_f^{n,\sharp}||_{0,K}) \nonumber\\
  &\leq C h_v ||v||_{1,\infty,\Omega_v} \sum_{K\in T_h}
  h_{K}^{k+1} h_{K_x}^{-1} ||f^{n,\sharp}||_{k+1,K}    || \xi_f^{n,\sharp}||_{0,K}
  \leq Ch^{k+1} || \xi_f^{n,\sharp}||_{0,\Omega}.
 \end{align}
 Here Cauchy-Schwarz inequality and approximate properties $\eqref{eq:approx}$ are used for the first and second inequality above, respectively.
 Applying the similar technique to the second term of $a_{h,1}(\eta_f^{n,\sharp};\xi_f^{n,\sharp})$, we get
\begin{align}
\label{3.6}
  & \left|\int_{T_h^v}\int_{\mathcal{E}_x} \left(\{\eta_f^{n,\sharp}v \}_x + \frac{|v\cdot n_x|}{2}[\eta_f^{n,\sharp}]_x   \right)\cdot [\xi_f^{n,\sharp}]_x ds_x dv\right| \nonumber\\
  &\leq \int_{T_h^v} \int_{\mathcal{E}_x}\left( |v\cdot n_x| (|\{ \eta_f^{n,\sharp} \}_x| + \frac{| [\eta_f^{n,\sharp}]_x  | }{2} )  \right) |[\xi_f^{n,\sharp}]_x| ds_x dv   \nonumber\\
  &\leq \left(\int_{T_h^v} \int_{\mathcal{E}_x} 2\left(  |\{ \eta_f^{n,\sharp} \}_x|^2 +  (\frac{| [\eta_f^{n,\sharp}]_x  | }{2})^2     \right) |v\cdot n_x|ds_x dv\right)^{\frac{1}{2}}  \left(\int_{T_h^v}\int_{\mathcal{E}_x} |v\cdot n_x| |[\xi_f^{n,\sharp}]_x|^2ds_x dv\right)^{\frac{1}{2}}  \nonumber \\
  &\leq C||\eta_{f}^{n,\sharp}||_{0,T_h^v \times \mathcal{E}_x}  \left(\int_{T_h^v}\int_{\mathcal{E}_x} |v\cdot n_x| |[\xi_f^{n,\sharp}]_x|^2ds_x dv\right)^{\frac{1}{2}}
  \leq Ch^{k+\frac{1}{2}} \left(\int_{T_h^v}\int_{\mathcal{E}_x} |v\cdot n_x| |[\xi_f^{n,\sharp}]_x|^2ds_x dv\right)^{\frac{1}{2}}.
\end{align}
By $\eqref{3.5}, \eqref{3.6}$ and Young's inequality, there is
\begin{eqnarray}
 \label{eq:3.2.3}
 a_{h,1}(\eta_f^{n,\sharp};\xi_f^{n,\sharp})
 &\leq& C|| \xi_f^{n,\sharp}||_{0,\Omega}^2 + C h^{2k+1} + \frac{1}{16}\int_{T_h^v}\int_{\mathcal{E}_x} |v\cdot n_x| |[\xi_f^{n,\sharp}]_x|^2ds_x dv.
\end{eqnarray}

\smallskip
\noindent{\sf Step 2: To estimate $a_{h,2}(\eta_f^{n,\sharp}, E_h^{n,\sharp}, B_h^{n,\sharp};\xi_f^{n,\sharp} ) $}.
Let $E_0 = \Pi_x^0 E^{n,\sharp}, B_0 = \Pi_x^0 B^{n,\sharp}$ be the $L^2$ projection of $E^{n,\sharp}, B^{n,\sharp}$ onto piecewise constant vector space with respect to $T_h^x$, then
\begin{align}
  \label{eq:3.2.1}
  &\left|\int_{\Omega} \eta_f^{n,\sharp} (E_h^{n,\sharp} + v\times B_h^{n,\sharp}) \cdot \nabla_v \xi_f^{n,\sharp}dx dv\right| \nonumber\\
  &= \left|\int_{\Omega} \eta_f^{n,\sharp} (E_h^{n,\sharp}-E_0 + v\times (B_h^{n,\sharp}-B_0) ) \cdot \nabla_v \xi_f^{n,\sharp}dx dv\right| \nonumber\\
  &\leq  (||E_h^{n,\sharp}-E_0||_{0,\infty,\Omega_x} + C||B_h^{n,\sharp}-B_0 ||_{0,\infty,\Omega_x} ) ||\eta_f^{n,\sharp}||_{0,\Omega} ||\nabla_v \xi_f^{n,\sharp}||_{0,\Omega}.
\end{align}
The first equality of $\eqref{eq:3.2.1}$ is satisfied because $(E_0 + v \times B_0)\cdot \nabla_v \xi_f^{n,\sharp} \in \mathcal{G}_h^k$, combining with the fact that the $L^2$ projection of $\eta_f^{n,\sharp}$ onto $\mathcal{G}_h^k$ is equal to zero.
Since the operator $\Pi_x^k$ is bounded in any $L^p$ $(1\leq p \leq \infty)$ norm \cite{CrouzeixThomee:1987}, we can further estimate $||E_h^{n,\sharp}-E_0||_{0,\infty,\Omega_x} \leq ||E_h^{n,\sharp}-\Pi_x^k E^{n,\sharp}||_{0,\infty,\Omega_x} + ||\Pi_x^k E^{n,\sharp} - E_0||_{0,\infty,\Omega_x}$, and $||\Pi_x^k E^{n,\sharp} - E_0||_{0,\infty,\Omega_x} = ||\Pi_x^k (E^{n,\sharp} -E_0) ||_{0,\infty,\Omega_x} \leq C||E^{n,\sharp}-E_0 ||_{0,\infty,\Omega_x} \leq C h_x ||E^{n,\sharp}||_{1,\infty,\Omega_x}$.
Thus, $  ||E_h^{n,\sharp}-E_0||_{0,\infty,\Omega_x} \leq ||\xi_{E}^{n,\sharp}||_{0,\infty,\Omega_x} + C h_x$. Similar treatment can be applied to $||B_h^{n,\sharp}-B_0||_{0,\infty,\Omega_x}$.
Therefore,
\begin{align}
 \label{3.7}
  &\left|\int_{\Omega} \eta_f^{n,\sharp} (E_h^{n,\sharp} + v\times B_h^{n,\sharp}) \cdot \nabla_v \xi_f^{n,\sharp}dx dv\right| \nonumber\\
  &\leq  ( ||\xi_{E}^{n,\sharp}||_{0,\infty,\Omega_x}  + ||\xi_{B}^{n,\sharp}||_{0,\infty,\Omega_x} + C h_x) ||\eta_f^{n,\sharp}||_{0,\Omega} ||\nabla_v \xi_f^{n,\sharp}||_{0,\Omega}
  \nonumber\\
  &\leq ( ||\xi_{E}^{n,\sharp}||_{0,\infty,\Omega_x}  + ||\xi_{B}^{n,\sharp}||_{0,\infty,\Omega_x} + C h_x)Ch^k || \xi_f^{n,\sharp}||_{0,\Omega}     \nonumber\\
  &\leq Ch^{k-\frac{d_x}{2}} (||\xi_{E}^{n,\sharp}||_{0,\Omega_x}  + ||\xi_{B}^{n,\sharp}||_{0,\Omega_x})||\xi_f^{n,\sharp}||_{0,\Omega}   + Ch^{k+1} || \xi_f^{n,\sharp}||_{0,\Omega}.
\end{align}
  Here we have applied the inverse inequality in $\eqref{eq:inv}$ to the last inequality above.
  The second term in $a_{h,2}(\eta_{f}^{n,\sharp}, E_h^{n,\sharp}, B_h^{n,\sharp}; \xi_f^{n,\sharp})$ can be estimated as
\begin{align}
  & \left|\int_{T_h^x}\int_{\mathcal{E}_v} \left(\{\eta_f^{n,\sharp}(E_h^{n,\sharp} + v\times B_h^{n,\sharp}) \}_v + \frac{|(E_h^{n,\sharp} + v\times B_h^{n,\sharp}) \cdot n_v|}{2}[\eta_f^{n,\sharp}]_v   \right)\cdot [\xi_f^{n,\sharp}]_v ds_v dx\right|
   \label{3.7.1}\\
    &\leq \int_{T_h^x} \int_{\mathcal{E}_v}\left( |(E_h^{n,\sharp} + v\times B_h^{n,\sharp}) \cdot n_v| (|\{ \eta_f^{n,\sharp} \}_v| + \frac{| [\eta_f^{n,\sharp}]_v  | }{2} )  \right)  |[\xi_f^{n,\sharp}]_v| ds_v dx   \nonumber\\
  &\leq \left(\int_{T_h^x} \int_{\mathcal{E}_v} 2|\{ (\eta_f^{n,\sharp})^2 \}_v|  | (E_h^{n,\sharp} + v\times B_h^{n,\sharp})\cdot n_v|ds_v dx\right)^{\frac{1}{2}}   
  \left(\int_{T_h^x}\int_{\mathcal{E}_v}  | (E_h^{n,\sharp} + v\times B_h^{n,\sharp})\cdot n_v| |[\xi_f^{n,\sharp}]_v|^2ds_v dx\right)^{\frac{1}{2}} \nonumber \\
  &\leq C||\eta_{f}^{n,\sharp}||_{0,T_h^x \times \mathcal{E}_v}  (||E_h^{n,\sharp}||_{0,\infty,\Omega_x}^{\frac{1}{2}} + ||B_h^{n,\sharp}||_{0,\infty,\Omega_x}^{\frac{1}{2}})
   \left(\int_{T_h^x}\int_{\mathcal{E}_v}  | (E_h^{n,\sharp} + v\times B_h^{n,\sharp})\cdot n_v| |[\xi_f^{n,\sharp}]_v|^2ds_v dx\right)^{\frac{1}{2}}
   \nonumber\\
  &\leq C h^{k+\frac{1}{2}}  (||E_h^{n,\sharp}||_{0,\infty,\Omega_x}^{\frac{1}{2}} + ||B_h^{n,\sharp}||_{0,\infty,\Omega_x}^{\frac{1}{2}})
   \left(\int_{T_h^x}\int_{\mathcal{E}_v}  | (E_h^{n,\sharp} + v\times B_h^{n,\sharp})\cdot n_v| |[\xi_f^{n,\sharp}]_v|^2ds_v dx\right)^{\frac{1}{2}}. \notag
\end{align}
Note $||E_h^{n,\sharp}||_{0,\infty,\Omega_x}^{\frac{1}{2}}
\leq ||\xi_E^{n,\sharp}||_{0,\infty,\Omega_x}^{\frac{1}{2}} + ||\Pi_x^k E^{n,\sharp}||_{0,\infty,\Omega_x}^{\frac{1}{2}}
\leq ||\xi_E^{n,\sharp}||_{0,\infty,\Omega_x}^{\frac{1}{2}} + C||E^{n,\sharp}||_{0,\infty,\Omega_x}^{\frac{1}{2}}$.
From the inverse inequality in $\eqref{eq:inv}$, there is $||\xi_E^{n,\sharp}||_{0,\infty,\Omega_x} \leq Ch^{-\frac{d_x}{2}} ||\xi_E^{n,\sharp}||_{0,\Omega_x}$. Since $||E^{n,\sharp}||_{0,\infty,\Omega_x}$ is bounded, $||E_h^{n,\sharp}||_{0,\infty,\Omega_x}^{\frac{1}{2}} \leq C(1 + h^{-\frac{d_x}{4}} ||\xi_E^{n,\sharp}||_{0,\Omega_x}^{\frac{1}{2}} )$.
A similar estimate also holds for $||B_h^{n,\sharp}||_{0,\infty,\Omega_x}^{\frac{1}{2}}$. Combing the results,
 we have
\begin{align}
\label{eq:3.8}
  ||E_h^{n,\sharp}||_{0,\infty,\Omega_x}^{\frac{1}{2}} + ||B_h^{n,\sharp}||_{0,\infty,\Omega_x}^{\frac{1}{2}}
    \leq C \left(1 + h^{-\frac{d_x}{4}}( ||\xi_E^{n,\sharp}||_{0,\Omega_x}^{\frac{1}{2}} + ||\xi_B^{n,\sharp}||_{0,\Omega_x}^{\frac{1}{2}}) \right).
\end{align}
Equations \eqref{3.7} - \eqref{eq:3.8} lead to
\begin{align*}
 &a_{h,2}(\eta_f^{n,\sharp},E_h^{n,\sharp},B_h^{n,\sharp};\xi_f^{n,\sharp})           \nonumber\\
 &\leq
 Ch^{k-\frac{d_x}{2}} (||\xi_{E}^{n,\sharp}||_{0,\Omega_x}  + ||\xi_{B}^{n,\sharp}||_{0,\Omega_x})||\xi_f^{n,\sharp}||_{0,\Omega}   + Ch^{k+1} || \xi_f^{n,\sharp}||_{0,\Omega} \nonumber\\
 &+
 Ch^{k+\frac{1}{2}}\left(1 + h^{-\frac{d_x}{4}}( ||\xi_E^{n,\sharp}||_{0,\Omega_x}^{\frac{1}{2}} + ||\xi_B^{n,\sharp}||_{0,\Omega_x}^{\frac{1}{2}}) \right)\left(\int_{T_h^x}\int_{\mathcal{E}_v}  | (E_h^{n,\sharp} + v\times B_h^{n,\sharp})\cdot n_v| |[\xi_f^{n,\sharp}]_v|^2ds_v dx\right)^{\frac{1}{2}}     \nonumber\\
 &\leq
 Ch^{2k-d_x} (||\xi_{E}^{n,\sharp}||_{0,\Omega_x}^2  + ||\xi_{B}^{n,\sharp}||_{0,\Omega_x}^2 ) + C ||\xi_f^{n,\sharp}||_{0,\Omega}^2   + Ch^{2k+2} \nonumber\\
 &+
 C h^{2k+1}\left(1 + h^{-\frac{d_x}{2}}( ||\xi_E^{n,\sharp}||_{0,\Omega_x} + ||\xi_B^{n,\sharp}||_{0,\Omega_x}) \right) 
 +
 \frac{1}{16}\int_{T_h^x}\int_{\mathcal{E}_v}  | (E_h^{n,\sharp} + v\times B_h^{n,\sharp})\cdot n_v| |[\xi_f^{n,\sharp}]_v|^2ds_v dx.  \nonumber
\end{align*}
Since $h^{2k+1} h^{-\frac{d_x}{2}}( ||\xi_E^{n,\sharp}||_{0,\Omega_x} + ||\xi_B^{n,\sharp}||_{0,\Omega_x})
 \leq C ( h^{4k+2-d_x} + ||\xi_E^{n,\sharp}||_{0,\Omega_x}^2 + ||\xi_B^{n,\sharp}||_{0,\Omega_x}^2)
  \leq C(h^{2k+2}+ ||\xi_E^{n,\sharp}||_{0,\Omega_x}^2 + ||\xi_B^{n,\sharp}||_{0,\Omega_x}^2)$ and $h^ {2k - d_x} \leq 1$ for $k \geq \left \lceil \frac{d_x}{2} \right \rceil$, we further have
\beq
 \label{eq:3.2.2}
  a_{h,2}(\eta_f^{n,\sharp},E_h^{n,\sharp},B_h^{n,\sharp};\xi_f^{n,\sharp})
 \leq  Ch^{2k+1} + C \ene^{n,\sharp}+ \frac{1}{16}\int_{T_h^x}\int_{\mathcal{E}_v}  | (E_h^{n,\sharp} + v\times B_h^{n,\sharp})\cdot n_v| |[\xi_f^{n,\sharp}]_v|^2ds_v dx.
\eeq
Now with the results in \eqref{eq:3.2.3} and \eqref{eq:3.2.2}, we can conclude \eqref{eq:3.1.2_1}.

\bigskip
\noindent\underline{To get \eqref{eq:3.1.2_2_0} and \eqref{eq:3.1.2_2}}, note that
  with $f$ being smooth with compact support in $\Omega_v$,
  one has $[f^{n,\sharp}]_v = 0$, $\{f^{n,\sharp}\}_v = f^{n,\sharp}$
  for any $e\in \mathcal{E}_v$. Thus, applying the divergence theorem gives
  \begin{eqnarray*}
  & &a_{h,2}(f^{n,\sharp},E_h^{n,\sharp},B_h^{n,\sharp};g)  \nonumber\\
  &=&
  \int_\Omega f^{n,\sharp}(E_h^{n,\sharp}+v\times B_h^{n,\sharp})\cdot \nabla_v g dxdv - \int_{\Omega_x}\int_{\mathcal{E}_v} f^{n,\sharp}(E_h^{n,\sharp}+v\times B_h^{n,\sharp})\cdot [g]_v ds_v dx                                           \nonumber\\
  &=& -\int_\Omega \nabla_v f^{n,\sharp}\cdot (E_h^{n,\sharp}+v\times B_h^{n,\sharp}) g dxdv.     \nonumber
  \end{eqnarray*}
  Similarly,  $a_{h,2}(f^{n,\sharp},E^{n,\sharp},B^{n,\sharp};g) = -\int_\Omega \nabla_v f^{n,\sharp}\cdot (E^{n,\sharp}+v\times B^{n,\sharp}) g dxdv$.
 Therefore,
  \begin{eqnarray}
   \label{eq:3.3.1}
  & &|a_h(f^{n,\sharp}, E^{n,\sharp}, B^{n,\sharp}; g) - a_h(f^{n,\sharp}, E^{n,\sharp}_h, B^{n,\sharp}_h; g)|        \nonumber\\
  &=& |a_{h,2}(f^{n,\sharp}, E^{n,\sharp}, B^{n,\sharp}; g) - a_{h,2}(f^{n,\sharp}, E^{n,\sharp}_h, B^{n,\sharp}_h; g)|   \nonumber\\
  &=& |\int_\Omega \nabla_v f^{n,\sharp}\cdot (E_h^{n,\sharp}-  E^{n,\sharp}+v\times (B_h^{n,\sharp}- B^{n,\sharp})) g dxdv | \nonumber\\
  &\leq& ||\nabla_v f^{n,\sharp}||_{0,\infty,\Omega} (||e_E^{n,\sharp}||_{0,\Omega} + C||e_B^{n,\sharp}||_{0,\Omega}) ||g||_{0,\Omega}  \notag\\
    &\leq& C (||e_E^{n,\sharp}||_{0,\Omega} + C||e_B^{n,\sharp}||_{0,\Omega}) ||g||_{0,\Omega}.  
      \end{eqnarray}
  This gives \eqref{eq:3.1.2_2_0}. Taking $g=\xi_f^{n,\sharp}$, and with the approximation property in \eqref{eq:approx}, we further get \eqref{eq:3.1.2_2},
  \begin{eqnarray*}
  & &|a_h(f^{n,\sharp}, E^{n,\sharp}, B^{n,\sharp}; \xi_f^{n,\sharp}) - a_h(f^{n,\sharp}, E^{n,\sharp}_h, B^{n,\sharp}_h; \xi_f^{n,\sharp})| \\
  &\leq&  C(||\xi_E^{n,\sharp}||_{0,\Omega_x} + ||\xi_B^{n,\sharp}||_{0,\Omega_x}+||\eta_E^{n,\sharp}||_{0,\Omega_x}+||\eta_B^{n,\sharp}||_{0,\Omega_x})  ||\xi_f^{n,\sharp}||_{0,\Omega}  \leq C h^{2k+2} + C \ene^{n,\sharp}.
  \end{eqnarray*}

\subsection{Proof of Lemma \ref{lemma:3.5_7}}
\label{proof:lemma:3.5_7}

\underline{To get \eqref{eq:lem:3.5_7.1}}, based on the definition of $a_h$, approximation property in \eqref{eq:approx} and inverse inequality \eqref{eq:inv},

\begin{eqnarray*}
  & & |a_h(\eta_f^{n,r}, E^{n,s+1}_h, B^{n,s+1}_h; g) - a_h(\eta_f^{n,r}, E^{n,s}_h, B^{n,s}_h; g)|              \nonumber\\
  &=& |a_{h,2}(\eta_f^{n,r}, E^{n,s+1}_h, B^{n,s+1}_h; g) - a_{h,2}(\eta_f^{n,r}, E^{n,s}_h, B^{n,s}_h; g)|    \nonumber\\
  &\leq& |\int_\Omega \eta_f^{n,r} (E_h^{n,s+1}-  E_h^{n,s}+v\times (B_h^{n,s+1}- B_h^{n,s})) \nabla_v g dxdv |         \nonumber\\
  & &+ |\int_{T_h^x} \int_{\mathcal{E}_v} \{\eta_f^{n,r}\}_v (E_h^{n,s+1}-  E_h^{n,s}+v\times (B_h^{n,s+1}- B_h^{n,s})) \cdot [g]_v   ds_v dx|
   \nonumber\\
  & &+ \frac{1}{2} \int_{T_h^x} \int_{\mathcal{E}_v} | [\eta_f^{n,r}]_v| \left| |(E_h^{n,s+1} + v\times B_h^{n,s+1}) \cdot n_v| - |(E_h^{n,s} + v\times B_h^{n,s}) \cdot n_v| \right|
   |[g]_v|   ds_v dx      \nonumber\\
  &\leq& C \left( ||E_h^{n,s+1} - E_h^{n,s}||_{0,\infty,\Omega_x} + ||B_h^{n,s+1} - B_h^{n,s}||_{0,\infty,\Omega_x} \right) (||\eta_f^{n,r}||_{0,\Omega}
   ||\nabla_v g||_{0,\Omega}+||\eta_f^{n,r}||_{0,T_h^x \times \mathcal{E}_v}
   ||g||_{0,T_h^x\times \mathcal{E}_v})\\
     &\leq& C h^k ||g||_{0,\Omega} \left( ||E_h^{n,s+1} - E_h^{n,s}||_{0,\infty,\Omega_x} + ||B_h^{n,s+1} - B_h^{n,s}||_{0,\infty,\Omega_x} \right).
\end{eqnarray*}
 By the a priori assumption and equation \eqref{eq:exact_new}, we have
 \begin{eqnarray}
 \label{inequalityE}
  & & ||E_h^{n,s+1} - E_h^{n,s}||_{0,\infty,\Omega_x}             \nonumber\\
  &=& ||E_h^{n,s+1} - E^{n,s+1} + E^{n,s} - E_h^{n,s} + E^{n,s+1} - E^{n,s}||_{0,\infty,\Omega_x}       \nonumber\\
  &\leq& ||e_E^{n,s+1}||_{0,\infty,\Omega_x} + ||e_E^{n,s}||_{0,\infty,\Omega_x}
  + ||E^{n,s+1} - E^{n,s}||_{0,\infty,\Omega_x} \leq C(h+\tau).
 \end{eqnarray}
Similarly, $||B_h^{n,s+1} - B_h^{n,s}||_{0,\infty,\Omega_x} \leq C(h+\tau)$. 
Therefore,
\begin{equation*}
|a_h(\eta_f^{n,r}, E^{n,s+1}_h, B^{n,s+1}_h; g) - a_h(\eta_f^{n,r}, E^{n,s}_h, B^{n,s}_h; g)|
\leq C(1 + \frac{\tau}{h})h^{k+ 1}  ||g||_{0,\Omega}.
\end{equation*}


\bigskip
\noindent\underline{To get \eqref{eq:lem:3.5_7.2}},  we follow the similar lines as to prove \eqref{eq:lem:3.5_7.1} and obtain
\begin{eqnarray*}
  & & |a_h(\xi_f^{n,r}, E^{n,s+1}_h, B^{n,s+1}_h; g) - a_h(\xi_f^{n,r}, E^{n,s}_h, B^{n,s}_h; g)|          \nonumber\\
  &\leq& C \left( ||E_h^{n,s+1} - E_h^{n,s}||_{0,\infty,\Omega_x} + ||B_h^{n,s+1} - B_h^{n,s}||_{0,\infty,\Omega_x} \right) (||\xi_f^{n,r}||_{0,\Omega}
   ||\nabla_v g||_{0,\Omega} + ||\xi_f^{n,r}||_{0,T_h^x \times \mathcal{E}_v}
   ||g||_{0,T_h^x\times \mathcal{E}_v})              \nonumber\\
  &\leq& C(h+\tau)  \left( ||\xi_f^{n,r}||_{0,\Omega} ||\nabla_v g||_{0,\Omega} +
  ||\xi_f^{n,r}||_{0,T_h^x \times \mathcal{E}_v} ||g||_{0,T_h^x \times \mathcal{E}_v} \right)                   
   \leq
   C(1 + \frac{\tau}{h}) ||\xi_f^{n,r}||_{0,\Omega} ||g||_{0,\Omega}
\end{eqnarray*}
Here \eqref{inequalityE} is used for the second inequality, and the inverse inequality \eqref{eq:inv} is used for the third one.

\subsection{Proof of Lemma \ref{lemma:3.9}}
\label{proof:lemma:3.9}

 By definition, we have
  \begin{eqnarray}
  &&\mathcal{L}_{RK}(g)
  = 2\mathcal{L}(g) - \mathcal{K}(g) - \mathcal{J}(g)                  \nonumber\\
  &=& \int_{T_h}\frac{3\eta_f^{n+1} + 3\eta_f^n - 6\eta_f^{n,2} + 3T_f^n(x,v)}{\tau} g dx dv
    + 2a_h(f^{n,2}, E^{n,2}, B^{n,2}; g) - 2a_h(f_h^{n,2}, E_h^{n,2}, B_h^{n,2}; g)            \nonumber\\
  & &- \left( a_h(f^{n,1}, E^{n,1}, B^{n,1}; g) - a_h(f_h^{n,1}, E_h^{n,1}, B_h^{n,1}; g) \right)        
  -
  \left( a_h(f^{n}, E^{n}, B^{n}; g) - a_h(f_h^{n}, E_h^{n}, B_h^{n}; g) \right)      \nonumber\\
  &=& \Lambda_1+\Lambda_2+\Lambda_3+\Lambda_4,\notag
 \end{eqnarray}
 where
 \begin{equation*}
 \Lambda_1=\frac{1}{\tau} \int_{T_h}\left(3\eta_f^{n+1} + 3\eta_f^n - 6\eta_f^{n,2} + 3T_f^n(x,v)\right) g dx dv,\\
 \end{equation*}
   \begin{eqnarray*}
  \Lambda_2
   & = &2a_h(f^{n,2}, E^{n,2}, B^{n,2}; g) - 2a_h(f^{n,2}, E_h^{n,2}, B_h^{n,2}; g) \nonumber\\
  &&- \left( a_h(f^{n,1}, E^{n,1}, B^{n,1}; g) - a_h(f^{n,1}, E_h^{n,1}, B_h^{n,1}; g) \right)  
  -
  \left( a_h(f^{n}, E^{n}, B^{n}; g) - a(f^{n}, E_h^{n}, B_h^{n}; g) \right),      \nonumber\\
 \Lambda_{3}&=& -2a_h(\eta_f^{n,2}, E_h^{n,2},B_h^{n,2};g) + a_h(\eta_f^{n,1}, E_h^{n,1},B_h^{n,1};g) + a_h(\eta_f^{n}, E_h^{n},B_h^{n};g),\\
 \Lambda_{4}&=&   2a_h(\xi_f^{n,2}, E_h^{n,2},B_h^{n,2};g) - a_h(\xi_f^{n,1}, E_h^{n,1},B_h^{n,1};g) - a_h(\xi_f^{n}, E_h^{n},B_h^{n};g).
 \end{eqnarray*}
 The term $\Lambda_1$ can be estimated by applying Lemmas~\ref{lemma:1} and~\ref{lemma:2},
 \begin{equation*}
  |\Lambda_1|
  \leq  \frac{1}{\tau}
  \left( ||3\eta_f^{n+1} + 3\eta_f^n - 6\eta_f^{n,2}||_{0,\Omega} + 3||T_f^n(x,v)||_{0,\Omega} \right) ||g||_{0,\Omega}
  \leq C(h^{k+1} + \tau^3) ||g||_{0,\Omega}.
 \end{equation*}
 To estimate $\Lambda_2$, we apply \eqref{eq:3.1.2_2_0} in Lemma \ref{3.1}, the approximation property \eqref{eq:approx}, and obtain
 \begin{equation*}
 |\Lambda_2|\leq C \sum_{\sharp=0}^2 \left(||e_E^{n,\sharp}||_{0,\Omega_x} + ||e_B^{n,\sharp}||_{0,\Omega_x}\right) ||g||_{0,\Omega}
 \leq C \sum_{\sharp=0}^2 \left(||\xi_E^{n,\sharp}||_{0,\Omega_x} + ||\xi_B^{n,\sharp}||_{0,\Omega_x}+h^{k+1}\right) ||g||_{0,\Omega}.
 \end{equation*}
 For $\Lambda_3$, we first rewrite it as below.
 \begin{eqnarray*}
  \Lambda_{3}
  &=&
  -2a_h(\eta_f^{n,2} - \eta_f^{n,1}, E_h^{n,2} , B_h^{n,2};g ) -
  a_h(\eta_f^{n,1} - \eta_f^{n}, E_h^{n} , B_h^{n};g )                         \nonumber\\
  & &-
  2a_h(\eta_f^{n,1}, E_h^{n,2} , B_h^{n,2};g) +
  2a_h(\eta_f^{n,1}, E_h^{n,1} , B_h^{n,1};g )                                 \nonumber\\
  & &-
  a_h(\eta_f^{n,1}, E_h^{n,1} , B_h^{n,1};g) +
  a_h(\eta_f^{n,1}, E_h^{n} , B_h^{n};g ).
 \end{eqnarray*}
 Applying Lemma~\ref{lemma:2} to the first line, and Lemma~\ref{lemma:3.5_7} - (1) to the second and the third lines,  one has
 \begin{equation*}
  \Lambda_{3} \leq C(1+\frac{\tau}{h}) h^{k+1} ||g||_{0,\Omega}.
 \end{equation*}
 For $\Lambda_4$, recall $G_2^n = 2\xi_f^{n,2} - \xi_f^{n,1} -\xi_f^{n}$, and using Lemma~\ref{lemma:3.5_7} - (2), we have
 \begin{eqnarray*}
  \Lambda_{4}
  &=&
    \left( 2a_h(\xi_f^{n,2}, E_h^{n,2}, B_h^{n,2}; g) - 2a_h(\xi_f^{n,2}, E_h^{n,1}, B_h^{n,1}; g) \right)
   + a_h(G_2^n, E_h^{n,1}, B_h^{n,1}; g)  \nonumber\\
  & &+
  \left( a_h(\xi_f^{n}, E_h^{n,1}, B_h^{n,1}; g) - a_h(\xi_f^{n}, E_h^{n}, B_h^{n}; g) \right)\notag\\
 &\leq&   C(1+\frac{\tau}{h}) \left( ||\xi_f^{n,2}||_{0,\Omega} +||\xi_f^{n}||_{0,\Omega}\right) ||g||_{0,\Omega}+ a_h(G_2^n, E_h^{n,1}, B_h^{n,1}; g).
 \end{eqnarray*}
 Now we combine the estimates for $\Lambda_i, i=1,\cdots 4$, and can get \eqref{eq:lem:3.9_1},
 \begin{align*}
 \mathcal{L}_{RK}(g) \leq
 C\left((1+\frac{\tau}{h})\sum_{\sharp=0}^{2} (||\xi_E^{n,\sharp}||_{0,\Omega_x} + ||\xi_B^{n,\sharp}||_{0,\Omega_x} + ||\xi_f^{n,\sharp}||_{0,\Omega}+h^{k+1}) +\tau^3 \right) ||g||_{0,\Omega} +a_h(G_2^n, E_h^{n,1}, B_h^{n,1}; g).
 \end{align*}
 To further bound the last term and hence obtain \eqref{eq:lem:3.9_2},
we apply the inverse inequality \eqref{eq:inv} and $||E_h^{n,1}||_{0,\infty,\Omega_x} + ||B_h^{n,1}||_{0,\infty,\Omega_x}$ being bounded  by
\eqref{eq:lem:6:proof}-\eqref{eq:lem:6:proof:1}, and get
 \begin{eqnarray*}
  &&\left|a_h(G_2^n, E_h^{n,1}, B_h^{n,1}; g) \right|
  \leq C(||G_2^n||_{0,\Omega} ||\nabla_x g||_{0,\Omega}+
   ||G_2^n||_{0,T_h^v \times \mathcal{E}_x} ||g||_{0,T_h^v \times \mathcal{E}_x})\notag\\
  &&\;\;\;\;+ C \left( ||E_h^{n,1}||_{0,\infty,\Omega_x} + ||B_h^{n,1}||_{0,\infty,\Omega_x}  \right)
  \left( ||G_2^n||_{0,\Omega} ||\nabla_v g||_{0,\Omega}  +
  ||G_2^n||_{0,T_h^x \times \mathcal{E}_v}
  ||g||_{0,T_h^x \times \mathcal{E}_v}  \right) \\
  &&\leq
  \frac{C}{h} ||G_2^n||_{0,\Omega} ||g||_{0,\Omega}.
 \end{eqnarray*}

 The estimate for $\mathcal{K}_{RK}(g)$ in  \eqref{eq:lem:3.9_3} can be proved very similarly.

\subsection{Proof of Lemma \ref{lemma:3.8}}
\label{proof:lemma:3.8}

 First we consider $a_{h,1}(G_1^n; G_2^n) + a_{h,1}(G_2^n; G_1^n)$.
 \begin{eqnarray}
  & &a_{h,1}(G_1^n; G_2^n) + a_{h,1}(G_2^n; G_1^n)                              \nonumber\\
  &=&
  \int_{\Omega} ( G_1^n v\cdot \nabla_x G_2^n + G_2^n v\cdot \nabla_x G_1^n ) dxdv
  - \int_{T_h^v} \int_{\mathcal{E}_x} (v\{G_1^n\}_x + \frac{|v\cdot n_x|}{2}[G_1^n]_x )\cdot [G_2^n]_x ds_x dv       \nonumber\\
  & &
  - \int_{T_h^v} \int_{\mathcal{E}_x} (v\{G_2^n\}_x + \frac{|v\cdot n_x|}{2}[G_2^n]_x )\cdot [G_1^n]_x ds_x dv       \nonumber\\
  &=& \int_{T_h^v}\int_{\mathcal{E}_x} v\cdot \left( [G_1^n G_2^n]_x - \{G_1^n\}_x [G_2^n]_x - \{G_2^n\}_x [G_1^n]_x \right) ds_x dv      
     - \int_{T_h^v}\int_{\mathcal{E}_x} |v\cdot n_x| [G_1^n]_x \cdot [G_2^n]_x ds_x dv                    \nonumber\\
  &=& - \int_{T_h^v}\int_{\mathcal{E}_x} |v\cdot n_x| [G_1^n]_x \cdot [G_2^n]_x ds_x dv.
 \end{eqnarray}
 The third equality above is due to $\eqref{eq:jump3}$.
 Hence,
 \begin{eqnarray}
 \label{eq:3.8.2}
 & &|a_{h,1}(G_1^n; G_2^n) + a_{h,1}(G_2^n; G_1^n)|
 \leq \int_{T_h^v}\int_{\mathcal{E}_x} |v\cdot n_x| \left( \frac{1}{32} |[G_1^n]_x|^2 +  8 |[G_2^n]_x|^2 \right) ds_x dv               \nonumber\\
 &\leq& \frac{1}{16} \int_{T_h^v}\int_{\mathcal{E}_x} |v\cdot n_x| \left( |[\xi_f^n]_x|^2 +  |[\xi_f^{n,1}]_x|^2 \right) ds_x dv + \frac{C}{h}||G_2^n||_{0,\Omega}^2.
 \end{eqnarray}
 The last inequality in \eqref{eq:3.8.2} is satisfied because of the inequality $|[G_1^n]_x|^2 = |[\xi_f^{n,1} - \xi_f^{n}]_x|^2 \leq 2(|[\xi_f^n]_x|^2 +  |[\xi_f^{n,1}]_x|^2)$
 and the inverse inequality in $\eqref{eq:inv}$.
 Similarly, we can also show
 \begin{eqnarray}
  \label{eq:3.8.1}
 & &\left|a_{h,2}(G_1^n, E_h^{n,1}, B_h^{n,1}; G_2^n) + a_{h,2}(G_2^n, E_h^{n,1}, B_h^{n,1}; G_1^n)\right|               \nonumber\\
 &=& \int_{T_h^x} \int_{\mathcal{E}_v} |(E_h^{n,1} + v \times B_h^{n,1})\cdot n_v| [G_1^n]_v \cdot [G_2^n]_v ds_v dx        \nonumber\\
 &\leq&  \int_{T_h^x} \int_{\mathcal{E}_v} |(E_h^{n,1} + v \times B_h^{n,1})\cdot n_v| |[\xi_f^{n,1}]_v|  |[G_2^n]_v| ds_v dx     \nonumber\\
 & &+ \int_{T_h^x} \int_{\mathcal{E}_v} |(E_h^{n,1} + v \times B_h^{n,1})\cdot n_v| |[\xi_f^{n}]_v|  |[G_2^n]_v| ds_v dx.
 \end{eqnarray}
 Denote each term on the right side of $\eqref{eq:3.8.1}$ as $\Lambda_1$ and $\Lambda_2$, respectively,
 then
 \begin{eqnarray}
 \label{eq:3.8.3}
  \Lambda_1 &\leq&  \int_{T_h^x} \int_{\mathcal{E}_v} |(E_h^{n,1} + v \times B_h^{n,1})\cdot n_v| \left( \frac{1}{16}|[\xi_f^{n,1}]_v|^2 +   4|[G_2^n]_v|^2 \right) ds_v dx  \nonumber\\
  &\leq& \frac{1}{16} \int_{T_h^x} \int_{\mathcal{E}_v} |(E_h^{n,1} + v \times B_h^{n,1})\cdot n_v| |[\xi_f^{n,1}]_v|^2 ds_v dx    
    +
  C ||E_h^{n,1} + v \times B_h^{n,1}||_{0,\infty,\Omega_x} \frac{||G_2^n||_{0,\Omega}^2 }{h}       \nonumber\\
  &\leq& \frac{1}{16} \int_{T_h^x} \int_{\mathcal{E}_v} |(E_h^{n,1} + v \times B_h^{n,1})\cdot n_v| |[\xi_f^{n,1}]_v|^2 ds_v dx  +
   \frac{C}{h} ||G_2^n||_{0,\Omega}^2 ,
 \end{eqnarray}
 here  we have used the fact that $||E_h^{n,1} + v \times B_h^{n,1}||_{0,\infty,\Omega_x} \leq C$  due to \eqref{eq:lem:6:proof} - \eqref{eq:lem:6:proof:1}.
 For the estimate of $\Lambda_2$,
 \begin{eqnarray}
  \Lambda_2
   &\leq&
  \int_{T_h^x} \int_{\mathcal{E}_v} \left|E_h^{n,1}-E_h^n + v \times (B_h^{n,1} - B_h^n )\right|  \left|[\xi_f^{n}]_v\right|  \left|[G_2^n]_v\right|  ds_v dx  \nonumber\\
  & &+  \int_{T_h^x} \int_{\mathcal{E}_v} \left|(E_h^{n} + v \times B_h^{n})\cdot n_v\right| \left|[\xi_f^{n}]_v\right| \left|[G_2^n]_v\right| ds_v dx          \nonumber\\
  &\leq&
  C \left( ||E_h^{n,1}-E_h^{n}||_{0,\infty,\Omega_x} +
   ||B_h^{n,1}-B_h^{n}||_{0,\infty,\Omega_x} \right)
    ||\xi_f^n||_{0,T_h^x \times \mathcal{E}_v}
    ||G_2^n||_{0,T_h^x \times \mathcal{E}_v}                     \nonumber\\
  & &+ \frac{1}{16} \int_{T_h^x} \int_{\mathcal{E}_v} \left|(E_h^{n} + v \times B_h^{n})\cdot n_v\right| \left|[\xi_f^{n}]_v\right|^2 ds_v dx  +
   \frac{C}{h}  ||G_2^n||_{0,\Omega}^2 .
 \end{eqnarray}
 As implied in the proof of Lemma~\ref{lemma:3.5_7}, $||E_h^{n,1}-E_h^{n}||_{0,\infty,\Omega_x} \leq C(h+\tau)$.
 So we further have
 \begin{eqnarray}
 \label{eq:3.8.4}
  \Lambda_2
  &\leq&
  C \left( 1 + \frac{\tau}{h}  \right)
    ||\xi_f^n||_{0,\Omega}    ||G_2^n||_{0,\Omega}   + \frac{C}{h}  ||G_2^n||_{0,\Omega}^2            
     + \frac{1}{16} \int_{T_h^x} \int_{\mathcal{E}_v} \left|(E_h^{n} + v \times B_h^{n})\cdot n_v\right| \left|[\xi_f^{n}]_v\right|^2 ds_v dx        \nonumber\\
  &\leq&   C \left( 1 + \frac{\tau}{h}  \right) ||\xi_f^n||_{0,\Omega}^2 +
   C \left( 1 + \frac{\tau}{h} + \frac{1}{h} \right) ||G_2^n||_{0,\Omega}^2               
   + \frac{1}{16} \int_{T_h^x} \int_{\mathcal{E}_v} \left|(E_h^{n} + v \times B_h^{n})\cdot n_v\right| \left|[\xi_f^{n}]_v\right|^2 ds_v dx. \notag\\
 \end{eqnarray}
 Note that $1, \frac{\tau}{h} \leq \frac{1}{h}$, hence $1 + \frac{\tau}{h} + \frac{1}{h} \leq \frac{C}{h}$. We now combine
 \eqref{eq:3.8.2}, \eqref{eq:3.8.3} and \eqref{eq:3.8.4} and conclude this lemma.

\subsection{Proof of Lemma \ref{lemma:new4.1}}
We only consider the case when $\sharp = 0$. The proof for other cases follows the same line.
For part $(i)$, with divergence theorem and equality \eqref{eq:jump2},

\begin{align*}
&b_h(\xi_E^n, \xi_B^n, \xi_f^n; \xi_E^n, \xi_B^n)\\
&=\int_{\Omega_x} \xi_B^n \cdot (\nabla \times \xi_E^n) dx -
\int_{\Omega_x} \xi_E^n \cdot (\nabla \times \xi_B^n) dx -
\int_{\Omega_x}(\int_{\Omega_v} v\xi_f^n dv) \cdot \xi_E^n dx \\
 &\;\;\;+
\int_{\mathcal{E}_x} \left( \{\xi_B^n\}_x - \frac{1}{2}[\xi_E^n]_{tan} \right) \cdot [\xi_E^n]_{tan}ds_x -
\int_{\mathcal{E}_x} \left( \{\xi_E^n\}_x + \frac{1}{2}[\xi_B^n]_{tan} \right) \cdot [\xi_B^n]_{tan}ds_x \\
&=
\int_{\mathcal{E}_x} \left( [\xi_E^n \times \xi_B^n]_x + \{\xi_B^n\}_x [\xi_E^n]_{tan} - \{\xi_E^n\}_x [\xi_B^n]_{tan}  \right)ds_x -\int_{\Omega_x}(\int_{\Omega_v} v\xi_f^n dv) \cdot \xi_E^n dx  -
\frac{1}{2} \stab_{EB}^n \\
&=-\int_{\Omega_x}(\int_{\Omega_v} v\xi_f^n dv) \cdot \xi_E^n dx  -
\frac{1}{2} \stab_{EB}^n
 \leq C(||\xi_f^n||_{0,\Omega}^2 + ||\xi_E^n||_{0,\Omega_x}^2) -
\frac{1}{2} \stab_{EB}^n.
\end{align*}

For part $(ii)$,
since $\nabla \times \xi_E^n \in \mathcal{U}_h^k$, there is $\int_{\Omega_x} \eta_B^n \cdot (\nabla \times \xi_E^n) dx = 0$. Similarly,
$\int_{\Omega_x} \eta_E^n \cdot (\nabla \times \xi_B^n) dx = 0$. Then
\begin{align*}
&|b_h(\eta_E^n, \eta_B^n, \eta_f^n; \xi_E^n, \xi_B^n)|\\
&= |\int_{\Omega_x}(\int_{\Omega_v} v\eta_f^n dv) \cdot \xi_E^n dx+
\int_{\mathcal{E}_x} \left( \{\eta_B^n\}_x - \frac{1}{2}[\eta_E^n]_{tan} \right) \cdot [\xi_E^n]_{tan}ds_x
-
\int_{\mathcal{E}_x} \left( \{\eta_E^n\}_x + \frac{1}{2}[\eta_B^n]_{tan} \right) \cdot [\xi_B^n]_{tan}ds_x|\\
 &\leq C ||\eta_f^n||_{0,\Omega}||\xi_E^n||_{0,\Omega_x}+
 C(||\eta_E^n||_{0,\mathcal{E}_x} +  ||\eta_B^n||_{0,\mathcal{E}_x})
 \left( \int_{\mathcal{E}_x} |[\xi_B^n]_{tan}|^2 + |[\xi_E^n]_{tan}|^2 ds_x \right)^{\frac{1}{2}} \\
& \leq
 Ch^{2k+1} + C||\xi_E^n||^2_{0,\mathcal{E}_x}+
 \frac{1}{16} \stab_{EB}^n.
\end{align*}

\subsection{Proof of Lemma \ref{lemma:4.4}}
\label{proof:lemma:4.4}

Based on definitions of $\mathcal{S}_{RK}(U, V)$, $X_2^n, Z_2^n$, and $G_2^n$, in addition to  Lemma \ref{lemma:1} and Lemma \ref{lemma:2} - (1), we have
 \begin{eqnarray}
  \label{eq:M15}
  & & \mathcal{S}_{RK}(U, V) =2 \mathcal{S}(U, V) - \mathcal{R}(U, V) -  \mathcal{Q}(U, V) \nonumber\\
  &\leq&
  \left( \frac{3\eta_E^{n+1} -6\eta_E^{n,2} + 3\eta_E^{n} + 3T_E^n(x) }{\tau}, U \right)_{\Omega_x}
   +
  \left( \frac{3\eta_B^{n+1} -6\eta_B^{n,2} + 3\eta_B^{n} + 3T_B^n(x) }{\tau}, V \right)_{\Omega_x} \nonumber\\
  && + b_h(2e_E^{n,2}-e_E^{n,1}-e_E^n, 2e_B^{n,2}-e_B^{n,1}-e_B^n, 2e_f^{n,2}-e_f^{n,1}-e_f^n; U, V)
     \nonumber\\
  &\leq &
   C(h^{k+1} + \tau^3)(||U||_{0,\Omega_x} + ||V||_{0,\Omega_x})
   + b_h(X_2^n, Z_2^n, G_2^n; U, V) - b_h(d_E^n, d_B^n, d_f^n; U, V),
 \end{eqnarray}
  where $d_\star^n = 2\eta_\star^{n,2}-\eta_\star^{n,1}-\eta_\star^n$ with $\star=E, B, f$.
 We further apply \eqref{eq:2.2.1_4} in Lemma \ref{lemma:2}, and this gives \eqref{eq:lem4:4_1},
 \begin{eqnarray}
    \mathcal{S}_{RK}(U, V)
   \leq
      C(h^{k+1} + \tau h^k+ \tau^3)(||U||_{0,\Omega_x} + ||V||_{0,\Omega_x})
   + b_h(X_2^n, Z_2^n, G_2^n; U, V).
 \end{eqnarray}
To further obtain \eqref{eq:lem4:4_2}, we need to estimate $b_h(X_2^n, Z_2^n, G_2^n; U, V)$.
  \begin{eqnarray}
    & & b_h( X_2^n, Z_2^n, G_2^n; U, V)  \nonumber\\
    &=&
      \int_{\Omega_x} Z_2^n \cdot (\nabla \times U) dx -  \int_{\Omega_x} X_2^n \cdot (\nabla \times V) dx
      - \int_{\Omega_x} (\int_{\Omega_v} G_2^n v dv) \cdot U dx
      \nonumber\\
    & &+
      \int_{\mathcal{E}_x} \left( \{Z_2^n \}_x - \frac{1}{2}[ X_2^n ]_{tan}\right) \cdot [U]_{tan} ds_x
       -
      \int_{\mathcal{E}_x} \left( \{X_2^n \}_x + \frac{1}{2}[ Z_2^n ]_{tan}\right) \cdot [V]_{tan} ds_x   \nonumber\\
    &\leq&
      ||Z_2^n||_{0,\Omega_x} ||\nabla \times U ||_{0,\Omega_x} + ||X_2^n||_{0,\Omega_x} ||\nabla \times V ||_{0,\Omega_x} + C||G_2^n||_{0,\Omega} ||U||_{0,\Omega_x}
      \nonumber\\
    & &+
      C( ||Z_2^n||_{0,\mathcal{E}_x} + ||X_2^n||_{0,\mathcal{E}_x} ) ( ||U||_{0,\mathcal{E}_x} + ||V||_{0,\mathcal{E}_x} ).
  \end{eqnarray}
Now we can apply inverse equalities in \eqref{eq:inv}, get
  $$ b_h(X_2^n, Z_2^n, G_2^n; U, V)
                \leq
               \frac{C}{h}(||Z_2^n||_{0,\Omega_x} + ||X_2^n||_{0,\Omega_x})(||U||_{0,\Omega_x} + ||V||_{0,\Omega_x}) + C||G_2^n||_{0,\Omega} ||U||_{0,\Omega_x},
$$
hence \eqref{eq:lem4:4_2}.

Similarly, one can get the estimate \eqref{eq:lem4:4_3} for $\mathcal{R}_{RK}(U, V)$.

\subsection{Proof of Lemma \ref{lemma:new4.2}}

From the definition of $b_h$,
  \begin{eqnarray}
     & & b_h( X_1^n, Z_1^n, G_1^n; X_2^n, Z_2^n) + b_h( X_2^n, Z_2^n, G_2^n; X_1^n, Z_1^n)  \nonumber\\
     &=&
        \int_{\Omega_x} \left( Z_1^n \cdot (\nabla \times X_2^n) dx -   X_1^n \cdot (\nabla \times Z_2^n)  +
                               Z_2^n \cdot (\nabla \times X_1^n) dx -   X_2^n \cdot (\nabla \times Z_1^n) \right) dx   \nonumber\\
     & &+
        \int_{\mathcal{E}_x} \left( \{Z_1^n \}_x - \frac{1}{2}[ X_1^n ]_{tan}\right) \cdot [X_2^n]_{tan} ds_x
        -
        \int_{\mathcal{E}_x} \left( \{X_1^n \}_x + \frac{1}{2}[ Z_1^n ]_{tan}\right) \cdot [Z_2^n]_{tan} ds_x   \nonumber\\
     & &+
        \int_{\mathcal{E}_x} \left( \{Z_2^n \}_x - \frac{1}{2}[ X_2^n ]_{tan}\right) \cdot [X_1^n]_{tan} ds_x
        -
        \int_{\mathcal{E}_x} \left( \{X_2^n \}_x + \frac{1}{2}[ Z_2^n ]_{tan}\right) \cdot [Z_1^n]_{tan} ds_x   \nonumber\\
     & &-
        \int_{\Omega_x} \left( \int_{\Omega_v} G_1^n v dv \right) \cdot X_2^n dx
        - \int_{\Omega_x} \left( \int_{\Omega_v} G_2^n v dv \right) \cdot X_1^n dx.\notag
  \end{eqnarray}
Using divergence theorem and equality $\eqref{eq:jump2}$, in addition to inverse inequality \eqref{eq:inv} and Young's inequality,
  we can simplify the inequality above as
 \begin{eqnarray}
   & & b_h( X_1^n, Z_1^n, G_1^n; X_2^n, Z_2^n) + b_h( X_2^n, Z_2^n, G_2^n; X_1^n, Z_1^n)  \nonumber\\
   &=&
   -\int_{\mathcal{E}_x} \left(  [X_1^n]_{tan} \cdot [X_2^n]_{tan} + [Z_1^n]_{tan} \cdot [Z_2^n]_{tan} \right)  ds_x
   -\int_{\Omega} \left(  G_1^n v \cdot X_2^n +  G_2^n v \cdot X_1^n \right) dx dv  \notag\\
  &\leq&
  \int_{\mathcal{E}_x} \left(
  \frac{1}{32}  |[X_1^n]_{tan}|^2 +
  8 |[X_2^n]_{tan}|^2 +
  \frac{1}{32}  |[Z_1^n]_{tan}|^2 +
  8 |[Z_2^n]_{tan}|^2  \right) ds_x
   + C \sum_{j=1}^2 ( ||G_j^n||_{0,\Omega}^2 + ||X_j^n||_{0,\Omega_x}^2 )
  \nonumber\\
  &\leq&
  \frac{C}{h}( ||X_2^n||_{0,\Omega_x}^2 + ||Z_2^n||_{0,\Omega_x}^2)
  +
  C \sum_{j=1}^2 ( ||G_j^n||_{0,\Omega}^2 + ||X_j^n||_{0,\Omega_x}^2 )+\frac{1}{16} \left( \stab_{EB}^n + \stab_{EB}^{n,1} \right).
  \label{eq:M1000}
\end{eqnarray}
 The last inequality of \eqref{eq:M1000} is due to that
 \begin{equation*}
 |[Z_1^n]_{tan}|^2 = |[\xi_B^{n,1} - \xi_B^{n}]_{tan}|^2 \leq 2|[\xi_B^{n,1}]_{tan}|^2 +  2 |[\xi_B^{n}]_{tan}|^2,
 \end{equation*}
 and a similar bound for $|[X_1^n]_{tan}|^2$.

\subsection{Proof of Lemma \ref{lemma:5.1}}
\label{proof:lemma:5.1}
  For any $g \in \mathcal{G}_h^k$, we first consider
  $a_h(f^{n,\sharp}, E^{n,\sharp}, B^{n,\sharp}; g) - a_h(f_h^{n,\sharp}, E_h^{n,\sharp}, B_h^{n,\sharp}; g)$, $\sharp=0, 1, 2$,
  \begin{eqnarray}
   & &a_h(f^{n,\sharp}, E^{n,\sharp}, B^{n,\sharp}; g) - a_h(f_h^{n,\sharp}, E_h^{n,\sharp}, B_h^{n,\sharp}; g) \nonumber\\
   &=&
    -a_h(\eta_f^{n,\sharp}, E_h^{n,\sharp}, B_h^{n,\sharp}; g) +
    ( a_h(f^{n,\sharp}, E^{n,\sharp}, B^{n,\sharp}; g) - a_h(f^{n,\sharp}, E_h^{n,\sharp}, B_h^{n,\sharp}; g) ) 
    +a_h(\xi_f^{n,\sharp}, E_h^{n,\sharp}, B_h^{n,\sharp}; g). \nonumber
  \end{eqnarray}
 With Cauchy-Schwarz inequality, approximation property in \eqref{eq:approx},  inverse inequality \eqref{eq:inv}, boundedness of  $||E_h^{n,1}||_{0,\infty,\Omega_x} + ||B_h^{n,1}||_{0,\infty,\Omega_x}$ in \eqref{eq:lem:6:proof} and \eqref{eq:lem:6:proof:1}, we have
 \begin{eqnarray}
  &&\left|a_h(\eta_f^{n,\sharp}, E_h^{n,\sharp}, B_h^{n,\sharp}; g)\right|
  \leq C(||\eta_f^{n,\sharp}||_{0,\Omega}||\nabla_x g||_{0,\Omega} +
   ||\eta_f^{n,\sharp}||_{0,T_h^v \times \mathcal{E}_x}
   ||g||_{0,T_h^v \times \mathcal{E}_x})\notag\\
  &&\;\;\;\; +C(||E_h^{n,\sharp}||_{0,\infty,\Omega_x} + ||B_h^{n,\sharp}||_{0,\infty,\Omega_x} )
    (||\eta_f^{n,\sharp}||_{0,\Omega} ||\nabla_v g||_{0,\Omega}+||\eta_f^{n,\sharp}||_{0,T_h^v \times \mathcal{E}_x}
   ||g||_{0,T_h^v \times \mathcal{E}_x})\notag \\
   &&\leq C h^k ||g||_{0,\Omega} ,
   \label{eq:VM3}
 \end{eqnarray}
 \begin{eqnarray}
  &&\left|a_h(\xi_f^{n,\sharp}, E_h^{n,\sharp}, B_h^{n,\sharp}; g)\right|
  \leq C(||\xi_f^{n,\sharp}||_{0,\Omega}||\nabla_x g||_{0,\Omega} +
   ||\xi_f^{n,\sharp}||_{0,T_h^v \times \mathcal{E}_x}
   ||g||_{0,T_h^v \times \mathcal{E}_x})\notag\\
  &&\;\;\;\; +C(||E_h^{n,\sharp}||_{0,\infty,\Omega_x} + ||B_h^{n,\sharp}||_{0,\infty,\Omega_x} )
    (||\xi_f^{n,\sharp}||_{0,\Omega} ||\nabla_v g||_{0,\Omega}+||\xi_f^{n,\sharp}||_{0,T_h^v \times \mathcal{E}_x}
   ||g||_{0,T_h^v \times \mathcal{E}_x}) \notag \\
   &&\leq \frac{C}{h} ||\xi_f^{n,\sharp}||_{0,\Omega}  ||g||_{0,\Omega}.
   \label{eq:VM4}
 \end{eqnarray}

  Following the derivation to get \eqref{eq:3.3.1} in Lemma~\ref{3.1}, we have
  \begin{eqnarray}
  \label{eq:VM5}
   & &a_h(f^{n,\sharp}, E^{n,\sharp}, B^{n,\sharp}; g) - a_h(f^{n,\sharp}, E_h^{n,\sharp}, B_h^{n,\sharp}; g) \nonumber\\
   &=&
   \int_{\Omega} \nabla_v f^{n,\sharp} \cdot
   \left(E_h^{n,\sharp} - E^{n,\sharp}  + v\times (B_h^{n,\sharp} -
   B^{n,\sharp})\right) g dx dv  \nonumber\\
   &\leq&
   C(||e_E^{n,\sharp}||_{0,\Omega_x} + ||e_B^{n,\sharp}||_{0,\Omega_x} )
   ||g||_{0,\Omega}
   \leq
   C(||\xi_E^{n,\sharp}||_{0,\Omega_x} + ||\xi_B^{n,\sharp}||_{0,\Omega_x}
   + h^{k+1})
   ||g||_{0,\Omega}.
  \end{eqnarray}
  Equations \eqref{eq:VM3} - \eqref{eq:VM5} lead to
  \begin{eqnarray}
  \label{eq:VM7}
   & &a_h(f^{n,\sharp}, E^{n,\sharp}, B^{n,\sharp}; g) - a_h(f_h^{n,\sharp}, E_h^{n,\sharp}, B_h^{n,\sharp}; g) \nonumber\\
   &\leq&
   C(||\xi_E^{n,\sharp}||_{0,\Omega_x} + ||\xi_B^{n,\sharp}||_{0,\Omega_x}
   + h^{k})
   ||g||_{0,\Omega}
   +
   \frac{C}{h} ||\xi_f^{n,\sharp}||_{0,\Omega}  ||g||_{0,\Omega}.
  \end{eqnarray}
  Now based on \eqref{eq:VM7} and Lemma~\ref{lemma:2}, we can bound $\mathcal{J}(\xi_f^{n,1})$.
  \begin{eqnarray}
  \label{eq:VM6}
   \mathcal{J}(\xi_f^{n,1})
   &=&
   \left( \frac{\eta_f^{n,1} - \eta_f^n}{\tau}, \xi_f^{n,1}  \right)_{\Omega}
   + a_h(f^{n}, E^{n}, B^{n}; \xi_f^{n,1}) - a_h(f_h^{n,}, E_h^{n}, B_h^{n}; \xi_f^{n,1})
   \nonumber\\
   &\leq&
      C(||\xi_E^{n}||_{0,\Omega_x} + ||\xi_B^{n}||_{0,\Omega_x}
   + h^{k})
   ||\xi_f^{n,1}||_{0,\Omega}
   +
   \frac{C}{h} ||\xi_f^{n}||_{0,\Omega}  ||\xi_f^{n,1}||_{0,\Omega}.
  \end{eqnarray}
   Based on equations $\eqref{eq:VM6}$ and $\eqref{eq:err1}$, we get
   \begin{eqnarray}
    ||\xi_f^{n,1}||_{0,\Omega}^2
    &=&
    (\xi_f^{n,1}, \xi_f^{n}) _{\Omega}
    + \tau {\mathcal J}(\xi_f^{n,1})  \nonumber\\
    &\leq&
    C\tau(||\xi_E^{n}||_{0,\Omega_x} + ||\xi_B^{n}||_{0,\Omega_x}
    + h^{k})
    ||\xi_f^{n,1}||_{0,\Omega}
    +
    (1+C\frac{\tau}{h}) ||\xi_f^{n}||_{0,\Omega}  ||\xi_f^{n,1}||_{0,\Omega},
   \nonumber
   \end{eqnarray}
   Cancel $||\xi_f^{n,1}||_{0,\Omega}$ from both sides of the inequality, and take the square of both sides,
   we have
   \begin{eqnarray}
    ||\xi_f^{n,1}||_{0,\Omega}^2
    &\leq&
    C\tau^2(||\xi_E^{n}||_{0,\Omega_x}^2 + ||\xi_B^{n}||_{0,\Omega_x}^2
    + h^{2k})
    +
    (1+C\frac{\tau}{h})^2 ||\xi_f^{n}||_{0,\Omega}^2.
   \end{eqnarray}
   By further using $\tau \leq \alpha h$, we obtain the upper bound of $||\xi_f^{n,1}||_{0,\Omega}^2$.
   Similarly,  $||\xi_f^{n,2}||_{0,\Omega}^2$ can be estimated based on \eqref{eq:VM7} and \eqref{eq:err1}.
   Likewise, we can establish the estimates for $||\xi_B^{n,\sharp}||_{0,\Omega_x}, ||\xi_E^{n,\sharp}||_{0,\Omega_x}$ for $\sharp = 1,2$.

\section{Extension and Conclusion}
\label{sec:5}

In this paper we prove the error estimates of fully discrete methods, which involve a third order Runge-Kutta time discretization and
upwind DG discretizations of
arbitrary order of accuracy in phase domain, for solving the  Vlasov-Maxwell system.
When the exact solutions have enough regularity, we show that the $L^2$ errors of the numerical solutions by such methods are
of $O(h^{k+\frac{1}{2}} + \tau^3)$ for $ k \geq \left \lceil \frac{d_x + 1}{2} \right \rceil$.
The third order Runge-Kutta time integration contributes to the error $O(\tau^3)$, while the error from the DG approximation
is $O(h^{k+\frac{1}{2}})$, which is expected for  hyperbolic systems with upwind numerical fluxes on general meshes.

The techniques used in this paper can be applied to the RKDG methods which involve other numerical fluxes, such as  central or alternating fluxes,
\begin{align*}
(\widehat{n_x\times E_h}, \widehat{n_x\times B_h})
&=(n_x\times \{E_h\}, n_x\times \{B_h\}),\qquad \textrm{(central)}\\
(\widehat{n_x\times E_h}, \widehat{n_x\times B_h})&=n_x\times (E_h^-, B_h^+), \textrm{or}\; n_x\times (E_h^+, B_h^-),\qquad\textrm{(alternating)},
\end{align*}
in the Maxwell solver. It was shown that these fluxes will result better energy conservation in semi-discrete DG methods \cite{CGLM}. On the other hand,
for RKDG methods with such fluxes, the stabilization mechanism $\stab_{EB}^{n,\sharp}$, $\sharp=0,1,2$,
in the form of  the tangential jump
$$
   \int_{\mathcal{E}_x} |[\xi_E^{n,\sharp}]_{tan}|_{0,\Omega_x}^2 + |[\xi_B^{n,\sharp}]_{tan}|_{0,\Omega_x}^2 ds_x
$$
 is no longer available from the Maxwell solver (see part $(i)$ of Lemma \ref{lemma:new4.1}), and this will lead to a sub-optimal $L^2$-norm error estimate: $C h^k+C\tau^3$.
With some insignificant modification to the details, almost the same error estimates can be established for the RKDG methods solving the
 smooth solutions of the relativistic Vlasov-Maxwell system of one species \cite{SS2010, YangLi:RVM},
\begin{equation*}
 \left\{
 \begin{aligned}
 &\partial_t f + \frac{v}{\sqrt{1+|v|^2}}\cdot \nabla_x f + (E+\frac{v}{\sqrt{1+|v|^2}}\times B)\cdot \nabla_v f = 0, \\
 &\partial_t E = \nabla \times B -J, \quad \partial_t B = -\nabla \times E,\\
 & \nabla \cdot E = \rho - \rho_i, \quad \nabla \cdot B = 0,
 \end{aligned}
 \right.
\end{equation*}
with
\begin{equation*}
  \rho(x,t)=\int_{\Omega_v}f(x,v,t)dv, \quad J(x,t)=\int_{\Omega_v}\frac{v}{\sqrt{1+|v|^2}}f(x,v,t)dv.
\end{equation*}


\end{document}